\documentclass[final]{siamart190516}

\usepackage[draft]{prelim2e} \usepackage{scrtime}

\usepackage[ngerman,english]{babel} 

\usepackage{epsfig,latexsym}
\usepackage{algorithm,algorithmicx,algpseudocode}
\usepackage{amsmath,amssymb}
\usepackage{amsfonts}
\usepackage{mathrsfs}
\usepackage{bbm}
\usepackage{tikz}
\usetikzlibrary{arrows} 
\usepackage[utf8]{inputenc}
\usepackage{color}
\usepackage{changes}

\newcommand{\TimesFont}{\RequirePackage{times}\RequirePackage[scaled=0.92]{helvet}}
\TimesFont

\newcommand{\dx}{\,\mathrm{d}x}
\newcommand{\domega}{\,\mathrm{d}\omega}
\renewcommand{\Pr}{\mathbb{P}}

\newcommand{\cost}[1]{\mathcal{C}\left[#1\right]}

\newcommand{\costb}[1]{\mathcal{C}\big[#1\big]}
\newtheorem{remark}[theorem]{Remark}
\newtheorem{assumption}[theorem]{Assumption}

\numberwithin{equation}{section}

\definecolor{correction}{rgb}{0,0,1}
\newcommand{\correction}[1]{{\color{black}{#1}}}
\newcommand{\update}[1]{{\color{black}{#1}}}

\begin{document}\selectlanguage{english}

\title{Adaptive multilevel subset simulation with selective refinement} 

\author{D. Elfverson\thanks{Formerly at Department of Mathematics and Mathematical statistics, Umeå University, SE-901 87 Umeå, Sweden (\texttt{daniel.elfverson@gmail.com}).}\and R. Scheichl\thanks{Institute for Applied Mathematics \& Interdisciplinary Center for Scientific Computing (IWR), Heidelberg University, D-69120 Heidelberg, Germany (\texttt{r.scheichl@uni-heidelberg.de}).} \and S. Weissmann\thanks{\correction{Institute of Mathematics, University of Mannheim,
  D-68138 Mannheim, Germany (\texttt{simon.weissmann@uni-mannheim.de})}.}\and F.A. DiazDelaO\thanks{Clinical Operational Research Unit, Department of Mathematics, University College London, London WC1H OBT, UK (\texttt{alex.diaz@ucl.ac.uk}).}
}

\date{}

\maketitle

\begin{abstract} 
In this work we propose an adaptive multilevel version of subset
  simulation to estimate the probability of rare events
  for complex physical systems. Given a sequence of nested failure domains 
  of increasing size, the rare event probability is expressed as a product of  
  conditional probabilities.
  The proposed new estimator uses different model resolutions and varying numbers of
  samples across the hierarchy of nested failure sets. In order to dramatically reduce the  computational 
cost, we construct the intermediate failure sets such that only a small number of expensive high-resolution model
  evaluations are needed, whilst the majority of
  samples can be taken from inexpensive low-resolution simulations. 
  A key idea in our new estimator is the use of a posteriori error estimators combined with a selective mesh refinement strategy
  to guarantee the critical subset property that may be violated when changing
  model resolution from one failure set to the next.  The efficiency
  gains and the statistical properties of the estimator are investigated both theoretically via shaking transformations, as well as
  numerically. On a model problem from subsurface flow, the new multilevel
  estimator achieves gains of more than a factor 60 over standard
  subset simulation for a practically relevant relative error of 25\%. 
\end{abstract}

\noindent {\footnotesize
\\
	\noindent {\bf AMS(MOS) subject classifications.} 65N30, 65C05, 65C40, 35R60
}

\section{Introduction} 
Estimating the probability of rare events  is one of the most important and
computationally challenging tasks in science and engineering. By definition, the
probability that a rare event occurs is very small. However, its effect could be catastrophic, e.g.,
when it is associated with some critical system failure such as the breakthrough
of pollutants into a water reservoir or the structural failure of an airplane wing. 
An efficient and reliable estimator is of utmost importance in such situations.

We are interested in estimating rare event probabilities occurring in
mathematical models of physical systems described by stochastic
differential equations
(SDEs) or partial differential equation (PDEs) with uncertain data, with
possibly large stochastic dimension. Standard Monte Carlo (MC) methods are
infeasible due to the huge amount of samples needed to produce even a single
rare event. Remedies for this are offered by different variance reduction techniques,
such as {\em importance sampling} \cite{SCHUELLER1987293,doi:10.1061,Bucher1988119},
{\em multilevel MC methods} \cite{Giles08,Barth2011,Cliffe2011,EHM14,Bucher2009504,MR2538857} and 
{\em subset simulation} \cite{AB01,AB03,PaWoKiZwSt14, UP2015}. In the statistics and
probability literature
\cite{SplittingMC}, subset simulation is also known as {\em splitting}
and can be interpreted as a sequential Monte Carlo method
\cite{Botev2010,MR2909622,GHM11,CAZT14,BLR15}. 
There are also Bayesian versions of subset simulation using Gaussian
process emulators to reduce the number of model
evaluations \cite{BLV16}. Moreover, a cross-entropy-based importance sampling method for rare events and a failure-informed dimension reduction through the connection to Bayesian inverse problems has been addressed in \cite{UPMS2021}.

In this work, we propose a multilevel version of the classical subset
simulation approach of Au \& Beck \cite{AB01} for rare event
estimation in the context of
complex models. {More specifically,} we combine the ideas of incorporating multilevel model resolutions across the subsets proposed in \cite{UP2015} with the selective mesh refinement strategy proposed in \cite{EHM14}.  As a result, our 
method does not violate the subset property, critical in subset simulation, when changing the model resolution. Furthermore, we formulate the subset simulation based on shaking {transformations to ensure asymptotic convergence of the resulting estimator, as shown in \cite{GL2015}}.

Let $(\Omega,\Sigma,\Pr)$ be \correction{a 
probability space with 
$\sigma$-algebra $\Sigma$ and probability measure $\Pr$, 
defined over some non-empty set $\Omega$. Given a mathematical model that describes the behaviour of a physical system, the sample space $\Omega$ will form the input space for any sources of uncertainty or randomness. Unless otherwise stated, we refer to a random variable $X$ with state space $\mathbb R^d$ as a $\Sigma$-$\mathcal B(\mathbb R^d)$-measurable mapping. The expected value of a measurable functional $F:\mathbb R^d\to\mathbb R^m$ of such a random variable $X$ with density 
$\pi$ is denoted by
\[\mathbb E[F(X)] = \int_{\Omega} F(X(\omega))\,\mathbb P({\mathrm d}\omega) = \int_{\mathbb R^d} F(x)\pi({\mathrm d}x) =: \mathbb E_{\pi}[F(X)].\]
If it is clear from the context, we will suppress the subscript $\pi$ and simply write $\mathbb E[F(X)]$. 
}

Let $F \in \Sigma$ 
be a \emph{rare event} that can be associated to system failure, that is, the demand exceeding the capacity 
of the system under study. Given the above probability space, subset simulation splits the estimation of $\Pr(F)  $ into 
a product of
conditional probabilities by introducing a sequence of nested events
\begin{equation}
\label{nested_events}
  F = F_K \subset F_{K-1} \subset \dots \subset F_1 \subset F_0 = \Omega,\quad F_j\in\Sigma,\ j=0,\dots,K,
\end{equation}
corresponding to larger failure probabilities as the subscript $j \to 0$.
The rare event probability $\Pr(F)$ is usually prohibitive to estimate directly by
MC sampling, but it can be expressed as the product of conditional probabilities
\begin{equation}
\label{subset_prod}
  \Pr(F) = \prod_{j=1}^K\Pr(F_j|F_{j-1}),
\end{equation}
where each $\Pr(F_j|F_{j-1})$ is larger than $\Pr(F)$ and hence easier to
estimate. The conditional probabilities in~\eqref{subset_prod}
can be estimated effectively using Markov Chain Monte Carlo (MCMC).

Computer simulations to study physical systems very often take the form of SDE or PDE models. Making use of a hierarchy of discretisations of these underlying models,
we exploit ideas from multilevel Monte Carlo (MLMC) methods
\cite{Giles08,Barth2011,Cliffe2011,EHM14} to reduce the overall cost
of subset simulation. We refer to \cite{HST2021,WLPU2020} for different variants of MLMC estimators of rare event probabilities. In principle, our approach is similar to the approach presented in \cite{UP2015}. However, in
contrast to that work and more closely related to the ideas in MLMC methods, 
we design the sequence of nested events in \eqref{nested_events} such
that most samples can be computed with inexpensive, coarse models. This
reduces the total cost significantly compared to classical subset
simulation, leading to more significant gains than the multilevel variant in \cite{UP2015}.

Following the ideas in \cite{EEHM14,EHM14}, another key advance in our new method is the use of a posteriori error
estimators to guarantee the critical subset property, which may be violated when
changing the model resolution from one intermediate failure set to the
next. It also allows a selective, sample-dependent choice of model
resolution. 
Finally, these three advances lead to 
significant speed-up over classical subset simulation for our multilevel
estimator. 

In summary, the paper makes the following contributions:
(i) \ The subset simulation method is formulated via shaking transformations, incorporating a selective mesh refinement strategy. The resulting algorithm is based on a MCMC method, where each sample involves an adaptive mesh resolution based on its limit state function value.  As a result, high accuracy is only needed for samples with a rather small limit state function, whereas for most samples low resolution evaluations are sufficiently accurate when the state is far away from failure.
(ii) \ Under certain assumptions, it is shown that the proposed selective refinement strategy does not violate the subset property. Moreover, a detailed complexity analysis quantifies the gains due to the selective refinement strategy.
(iii) \ A novel, adaptive multilevel subset simulation method is proposed where the accuracy increases over the defined subsets. Through a selective refinement strategy and appropriately chosen intermediate threshold values, the failure sets satisfy the critical subset property. Our complexity analysis shows significant improvement through the proposed adaptive multilevel subset choice and through the additional application of selective mesh refinement.

The paper is organized as follows. {In Section~\ref{problem-formulation}, we formulate
the problem and define hierarchies of discretisations via two abstract, sample-wise assumptions.
In Section~\ref{s:rare}, we explore in detail the estimation of rare event probabilities, first
using standard MC and classical subset simulation, before proposing two improved estimators 
based on selective refinement and adaptive multilevel subset selection. Section~\ref{sec:MCMC} contains a 
concrete implementation of the 
algorithms via shaking transformations, as well as
an asymptotic convergence result. The complexity analysis of the novel approaches is provided in Section \ref{sec:complexity},
while Section~\ref{s:numerics} demonstrates their performance on a series of numerical experiments, starting with two toy problems and finishing with a 
Darcy flow model. 
Finally, Section~\ref{s:conclusions} offers some conclusions.}

\section{Problem Formulation and Model \correction{Hierarchies}} \label{problem-formulation} 

We consider a (linear or nonlinear) model $\mathcal{M}$ on an infinite
dimensional function space $V$, e.g., a PDE, which is subject
to uncertainty, or an SDE. The solution is modelled as a random field on the probability space
$(\Omega,\Sigma,\Pr)$ with values in $V$. For
any $\omega\in\Omega$, we denote by $u(\omega) \in V$ the solution of 
\begin{equation}\label{eq:model}
	\mathcal{M}(\omega,u(\omega)) = 0.
\end{equation}
Given a quantity of interest $\Phi:V \to \mathbb{R}$, i.e., a functional of
the model solution $u$, we are interested in computing the probability
that a so-called {\em rare event} occurs. We let $G=G(\omega):=\Phi^*
- \Phi(u(\omega))$ be the associated limit state function, which is 
negative when the quantity of interest exceeds a critical value
$\Phi^*$. Thus, we want to compute the
probability that $\omega\in\Omega$ is in the failure set
\[F:=\{\omega\in\Omega:G(\omega)\leq 0\}.\]
For simplicity, we will assume that $G$ is a real valued random variable with probability density function (pdf) (with respect to Lebesgue measure) $\pi:\mathbb R\to \mathbb R_+$, which is assumed to be unknown.  If $\mathbbm{1}_F$ denotes the indicator function of $F$, i.e., $\mathbbm{1}_F(\omega)=1$ if
$\omega\in F$ and  $\mathbbm{1}_F(\omega)=0$ otherwise, then the failure
probability can also be expressed as the integral
\begin{equation}\label{eq:P}
{P=\Pr\left(F\right)=\int_\Omega \mathbbm{1}_F(\omega)\,\Pr(\domega) = \int_{\mathbb R} \mathbbm{1}_{(-\infty, 0]}(x)\pi(x)\, \dx = \int_{-\infty}^0 \pi(x)\, \dx}
\end{equation}
which, for very small $P\ll 1$, will be classified as a rare event. Note that the integral in \eqref{eq:P} 
is equivalent to the expected value \correction{$\mathbb{E}[{\mathbbm{1}_F}]$}. We are interested
in applications where not only the dimension of $V$, but also the
dimension of the underlying sample space $\Omega$ is high (or infinite),
e.g., in subsurface flow simulations where the permeability is described by a
spatially correlated random field.  To estimate $P$, both $V$ and the high-dimensional
integral in \eqref{eq:P} need to be approximated in practice.  

We return to the approximation of the above integral
in 
Section \ref{s:rare} and finish this section by formulating some abstract assumptions on the 
numerical approximation of the model $\mathcal{M}$ and of the limit
state function $G$, for any given $\omega \in \Omega$. To this end, 
we introduce a hierarchy of numerical approximations to $G$ with increasing accuracy, namely $G_\ell$, for $\ell 
= 1,\ldots,L$. 

\begin{assumption}\label{ass:error_estimate1}
{\correction{\begin{enumerate}
\item[(a)] We assume that the cdf of $G$ is Lipschitz continuous, i.e., there exists a constant $c_{\mathrm{Lip}}>0$ such that for any $x,y\in\mathbb R$
\begin{equation}
|\Pr(G\le y)-\Pr(G\le x)|\le c_{\mathrm{Lip}} |y-x|.
\end{equation}
\item[(b)] 
Let $\gamma\in(0,1)$ and $q \ge 0$. Then, for all $\ell \in \mathbb{N}$, we assume that
\begin{equation}
\label{eq:error_bound}
  |G(\omega) - G_\ell(\omega)| \leq \gamma^\ell, \quad 
  \Pr-\text{a.s. in} \ \ \omega,
\end{equation}
and that there exists a constant $c_0>0$ such that 
\begin{equation}\label{eq:cost1}
\correction{\mathbb E[\cost{G_\ell}]} \le c_0 \gamma^{-\ell q}.
\end{equation}
\end{enumerate}}}
\end{assumption}
 Note that \correction{to estimate $\Pr(F)$, high} accuracy is only needed for samples $\omega \in \Omega$ where the limit state function $G$ is small. \correction{Thus, given a family of random variables $(G_\ell,\; \ell\in\mathbb N)$ satisfying Assumption~\ref{ass:error_estimate1}, the following {\em selective refinement} 
 strategy was introduced in \cite{EEHM14,EHM14}, in fact for the more general task of approximating $\Pr(G\le y)$ for any $y\in\mathbb R$.
\begin{algorithm}[H]
\caption{Selective refinement strategy introduced in \cite{EEHM14,EHM14}.}
\label{alg:selective}
\begin{algorithmic}[1]
\correction{\State Let $y\in\mathbb R$, $\ell \in \mathbb N$ and $\omega\in\Omega$.
    \State Compute $G_{1}(\omega)$ and initialise $k=1$ and $G_\ell^y(\omega) = G_1(\omega)$.
    \While{$k < \ell$ and \update{$|G_{k}(\omega)-y| < \gamma^{k}$}} \State refine the accuracy by setting $k \leftarrow k+1$ and compute $G_{k}(\omega)$.
    \State Set $G_\ell^y(\omega) = G_{k}(\omega)$.
\EndWhile}
\end{algorithmic}
\end{algorithm}
It has been shown in \cite[Lem.~5.1 \& 5.2]{EHM14} that the construction in Algorithm \ref{alg:selective} leads to a ($y$-dependent) random variable $G_\ell^y:\Omega\to\mathbb R$ which satisfies the following assumption.} 
\correction{\begin{assumption}\label{ass:selective_refinement}
Suppose that Assumption~\ref{ass:error_estimate1} (a) holds and let $\gamma\in(0,1)$ and $q \ge 0$. Then, for fixed $y\in\mathbb R$ and
for all $\ell \in \mathbb{N}$, we assume that 
\begin{equation}
\label{eq:error_bound2}
  |G(\omega) - G_\ell^y(\omega)| \leq { \max\Big(\gamma^\ell,|G_\ell^y(\omega)-y|\Big)}, \quad 
  \Pr-\text{a.s. in} \ \ \omega,
\end{equation}
and that there exists a constant $c_0>0$ such that \begin{equation*}
\mathbb E[\cost{G_\ell^y}] 
\le c_0\left(1+\gamma^{(1-q)\ell}\right).
\end{equation*}

\end{assumption}
It was also shown in the proof of \cite[Lem.~5.2]{EHM14} that under Assumption~\ref{ass:selective_refinement} the bias for approximating $\Pr(G\le y)$ remains of order $\gamma^\ell$, i.e., there exists a constant $c_y>0$ independent of $\gamma$ and $\ell$ such that
\begin{equation*}
\left|\Pr(G_\ell^y\le y)-\Pr(G\le y)\right|\le c_y\gamma^{\ell}.
\end{equation*}
In order to verify that $G_\ell^y$ indeed is a random variable on $(\Omega,\Sigma,\mathbb P)$, we introduce (as in~\cite{Detommaso2019}) the stopping time $\omega\mapsto\tau_y(\omega) := \inf\{k\in \{1,\dots,\ell\}\mid  |G_k(\omega) -y|\ge \gamma^{\ell}\}$ with respect to the natural filtration $\mathcal F_k = \sigma(G_s,s\le k)$ and define $G_\ell^y:= G_{\ell\wedge\tau_y}$ as the final iteration of the stopped discrete time stochastic process $(G_{k\wedge \tau_y},k = 1,\dots,\ell)$ which is measurable with respect to $\mathcal F_L\subset \Sigma$, since all $G_k$ are assumed to be $\Sigma/\mathcal B(R)$-measurable.

Note that Assumption 2.2 holds naturally for the choice $y=\infty$, since any finite random variable $G_\ell$ satisfying Assumption 2.1 also satisfies \eqref{eq:error_bound2} and $\mathbbm{1}_{G_\ell\le \infty}=1$ almost surely.}
\begin{remark}\label{rem:notation}
For the complexity analysis of subset simulation and of the standard MC method, we will consider a fixed $y\in\mathbb R$ and thus suppress the dependence of $G_\ell^y$ on $y$. \correction{Under Assumption~\ref{ass:selective_refinement} we then refer to the random variable $G_\ell^y$ satisfying \eqref{eq:error_bound2} simply by $G_\ell$.}
\end{remark}

In a PDE setting, we typically have \correction{approximations \update{$\tilde G_h$} of $G$ associated with some numerical discretisation method and} error bounds of the type
\begin{equation}
\label{uniform_bound}
  |G(\omega)-\update{\tilde G_h(\omega)}|\leq C(\omega)\correction{h^\alpha}\,,
\end{equation}
where $h$ is a discretisation parameter, e.g., mesh size, and $\alpha$ is a convergence rate. The constant $C(\omega)$ may depend on $\omega$. for SDEs or PDEs with random coefficients, $\alpha$ and $C(\omega)$ can be estimated 
using adjoint methods or hierarchical error estimators, see, e.g., \cite{BDW2011,GS02,O2018}. 

\correction{For sufficiently well behaved models $\mathcal M$ and sufficiently smooth functionals $\Phi$, there exists a constant $C_{\max} < \infty$ such that $C(\omega) \le C_{\max}$ $\Pr$-almost surely in $\omega$. Thus, considering for example numerical approximations of $G$ obtained} on a sequence of uniformly refined meshes $\mathcal{T}_\ell$, $\ell \in \mathbb{N}_0$, with $h_\ell = h_0 2^{-\ell}$ and $\mathcal{T}_0$ some fixed coarsest mesh, the 
bound in \eqref{eq:error_bound} can be achieved
\correction{for every $\ell \in \mathbb N$ by choosing $G_\ell = \update{\tilde G_{h_\ell}}$ and $\gamma = 2^{-\alpha}$, provided $h_0 \le C_{\max}^{-1/\alpha}$.} 
The constant $q$ in \eqref{eq:cost1} depends on the
(physical) dimension of the problem, the order of accuracy of the
underlying numerical method, and the choice of deterministic
solver. The constant $c_0$ depends on the distribution of the constant $C(\omega)$. 
A similar choice is possible in the case of \correction{finite element methods with} locally adaptive mesh refinement
driven by an a posteriori error estimator. In that case, the bound in 
\eqref{uniform_bound} can be replaced by a sharper bound in
terms of the number of mesh elements, cf.\ \cite{BeRa96}. 

{\update{Finally, we mention that both Assumption~\ref{ass:error_estimate1} and Assumption~\ref{ass:selective_refinement} are formulated as $\mathbb P$-almost sure conditions, and we are aware of the fact that in practical scenarios this may be very restrictive. It is of great interest to relax this assumption and to allow numerical error bounds that hold only with high probability or in an $L^p$-sense instead of $\mathbb P$-almost surely. However, this relaxation would lead to non-trivial changes in the complexity analysis presented in Section~\ref{sec:complexity}. Therefore, we will leave this extension for future work.}}

\section{Adaptive Subset Simulation and Selective Refinement} 
\label{s:rare}

In this section, we explore in detail the estimation of the failure probability $P$ given by \eqref{eq:P}, {and propose two new extensions of the classical subset simulation approach.}

Let $\widehat{P}$ denote an estimator of $P$.  There are several sources of error in this 
estimator. These include: systematic error due to the choice of mathematical model, numerical error due to model 
approximation, and statistical error due to finite sampling size. 
Here, we will assume the model 
is 
exact and will only consider how to control the numerical and the statistical errors.  We measure the quality 
of the estimator $\widehat P$
by \correction{using the relative root mean squared error (rRMSE)} 
\begin{equation}\label{eq:def_cov}
\delta\big(\widehat P\big) =
\frac1{P}\sqrt{\correction{\mathbb E}[\big(\widehat P-P\big)^2}]
 \end{equation}
and we let $\text{TOL}>0$ be the desired accuracy for $\delta\big(\widehat P \big)$. \correction{In subset simulation, the rRMSE is often referred to as the coefficient of variation (c.o.v.), see e.g.~\cite{AB01}.}

\subsection{Estimating rare event probabilities by standard Monte Carlo}

The most basic way of estimating the probability in 
\eqref{eq:P} to an accuracy $\text{TOL}$ is to use a 
standard MC method where all samples are computed with a numerical approximation to accuracy $\mathcal{O}(\text{TOL})$. 
Let $F^L:=\{\omega\in\Omega:G_L(\omega)\leq 0\}$ be the approximate failure set on a fixed numerical discretisation level $L$, and consider the standard MC estimator 
\begin{equation}
  \widehat P^{\text{MC}} =\frac{1}{N}\sum_{i=1}^N \mathbbm{1}_{F^L} (\omega_i) = \frac{1}{N}\sum_{i=1}^N \mathbbm{1}_{(-\infty, 0]}(G_L^{(i)}) \,,
\end{equation}
where $\omega_i\in\Omega$ are independent and identically distributed (i.i.d.) samples and each $G_L^{(i)} = G_L(\omega_i)$ 
is computed to accuracy $\gamma^L$. This is an unbiased
estimator for 
\begin{equation}
\label{eq:PL}
P_L=\Pr\left(F^L\right)=\int_\Omega\mathbbm{1}_{F^L} (\omega)\Pr(\domega) = \correction{\mathbb E}\left[\widehat 
P^{\text{MC}}\right],
\end{equation}
and we can expand
\begin{equation}\label{expand_del}
\begin{split}
\delta\big(\widehat P^{\text{MC}}  \big)^2 &= {\frac{\correction{\mathbb E\left[(\widehat P^{\text{MC}})^2\right]}-2\correction{\mathbb E\left[\widehat P^{\text{MC}}\right]} P+ P^2}{P^2}}\\
&= {\frac{\correction{\mathbb E\left[(\widehat P^{\text{MC}})^2\right]}-\correction{\mathbb E\left[\widehat P^{\text{MC}}\right]}^2 + \correction{\mathbb E\left[\widehat P^{\text{MC}}\right]}^2 - 2\correction{\mathbb E\left[\widehat P^{\text{MC}}\right]} P+ P^2}{P^2}}\\
 &= \Bigg(\frac{ \correction{\mathbb E\left[\widehat P^{\text{MC}} - P\right]}}{P}\Bigg)^2 + \frac{\correction{\mathbb V}\left[\widehat P^{\text{MC}} \right]}{P_L^2} {\frac{P_L^2}{P^2}} \  = \ \left(1 -
  \frac{P_L}{P}\right)^2 + \frac{1-P_L}{N P_L}{\frac{P_L^2}{P^2}}\,,
 \end{split}
\end{equation}
which
includes a bias error due to the numerical
approximation $G_L \approx G$. \correction{The expectation and the variance of $\widehat P^{\text{MC}}$ in \eqref{expand_del} are with respect to the joint distribution of the samples $G_L^{(i)}$, in the i.i.d.~case here, the $N$-fold product measure of the distribution of $G_L$.}  As shown in
  \cite[Sec.~3]{EHM14}, it follows from
  Assumption~\ref{ass:selective_refinement} that there exists a constant $c_1>0$ such that
  \[
\left( 1 - \frac{P_L}{P} \right)^2 \le c_1^2 \gamma^{2L} \,. 
\]
\correction{To ensure that the first term on the
right hand side of \eqref{expand_del} is less than $\text{TOL}^2/2$} it suffices that
\begin{equation}
\label{def:L}
\correction{L 
\ge \log\left(\sqrt{2} c_1 \text{TOL}^{-1}\right)}\,.
\end{equation}
We will assume this throughout the paper.  

\correction{We note that the constant $c_1$ does in general depend on the underlying distribution of $G$ and $G_L$, in particular, on the rareness of the event and on the gradient of $G$ near} the boundary of the failure domain. \correction{However, $L$ only grows logarithmically with $c_1$ and this contribution to the rRMSE is the same for all methods considered.} We refer to \cite{WLPE2021} for a detailed analysis of approximation errors for rare event probabilities in the context of PDE based models.

The main challenge in achieving the required accuracy for
$\delta\big(\widehat P^{\text{MC}}  \big)$, is to ensure that 
the second term on the right 
hand side of \eqref{expand_del} is sufficiently small such that
\begin{equation}
  \frac{1-P_L}{N P_L} \frac{P_L^2}{P^2}\leq  \frac{\text{TOL}^2}{2}\,.
\end{equation}
A sufficient condition for this to hold is
\begin{equation}
   N \propto \textrm{TOL}^{-2} P_L^{-1}(\correction{P_L/P})^2 \,.
\end{equation}
Hence, the number of samples needs to be proportional to the inverse
of the rare event probability. For realistic applications, where the cost
$\cost{G_L} \gg 1$ and $P_L \approx P\ll1$, this is completely
infeasible. Under Assumption~\ref{ass:error_estimate1} and choosing
$L$ as in \eqref{def:L}, the total cost of the standard MC estimator would be
\begin{equation}
\label{eq:cost_MC}
  \correction{\mathbb E\left[\costb{\widehat P^{\text{MC}}}\right]}= N \correction{\mathbb E}\left[\cost{G_L}\right] \le  N c_0
  \gamma^{-Lq} 
   \le c_2 \text{TOL}^{-(2+q)} P_L^{-1}\correction{\frac{P_L^2}{P^2}}\,,
\end{equation}
{for some constant $c_2>0$ which is independent of $\text{TOL}$.}
It is possible to improve this through importance sampling techniques 
\cite{SCHUELLER1987293,doi:10.1061,Bucher1988119}, but that requires
some a priori knowledge of the distribution of $G_L$, which we
typically do not have.  Further note that the factor $\correction{P_L/P}$ is close to $1$,  since $L$ was chosen large enough such that $\correction{P_L/P \in (1-\text{TOL}/\sqrt{2},1+\text{TOL}/\sqrt{2})}$.

\subsection{Subset Simulation}
\label{sec:subset}

{In engineering applications, subset} simulation \cite{AB01,AB03,PaWoKiZwSt14} is one of the most widespread
variance reduction techniques to design an efficient estimator for $P$ in Equation~\eqref{eq:P}. It has been successfully used in many different contexts and for different applications, which include engineering reliability analysis \cite{AuBook}, robust design \cite{Au05}, topology optimisation \cite{QiLu11}, multi-objective optimisation \cite{XiXi17}, Bayesian inference \cite{DiGa17} and model calibration through history matching \cite{Gong2021}. 

The main idea is to define a sequence of nested 
failure sets that 
contain the target failure set $F$, as in~\eqref{nested_events}.
This is accomplished using a sequence of intermediate failure thresholds
\begin{equation}
  0 = y_K < y_{K-1} < \dots < y_0 = \infty.
\end{equation}
That way, each intermediate failure set is defined as 
\[
F_j:=\{\omega\in\Omega:G(\omega)\leq y_j\} = \{ G\in (-\infty, y_j]\}, \quad \text{for} \ \ j = 1,\ldots,K.
\]
The failure set $F$ and some intermediate failure sets that contain it are illustrated in 
Fig.~\ref{fig:classical_subset} via level curves of $G(\omega)$.
As stated in \eqref{subset_prod}, the rare event probability $\Pr(F)$ can be expressed
as product of conditional probabilities, i.e., 
$\Pr(F) = \prod_{j=1}^K\Pr(F_j|F_{j-1})$,
where each $\Pr(F_j|F_{j-1})$ is by construction larger than $\Pr(F)$ \correction{and thus less rare} -- significantly so, if $K$ is sufficiently large \correction{and the $y_j$ are chosen appropriately}.
\begin{figure}[t]
\centering
\includegraphics[width=0.5\textwidth]{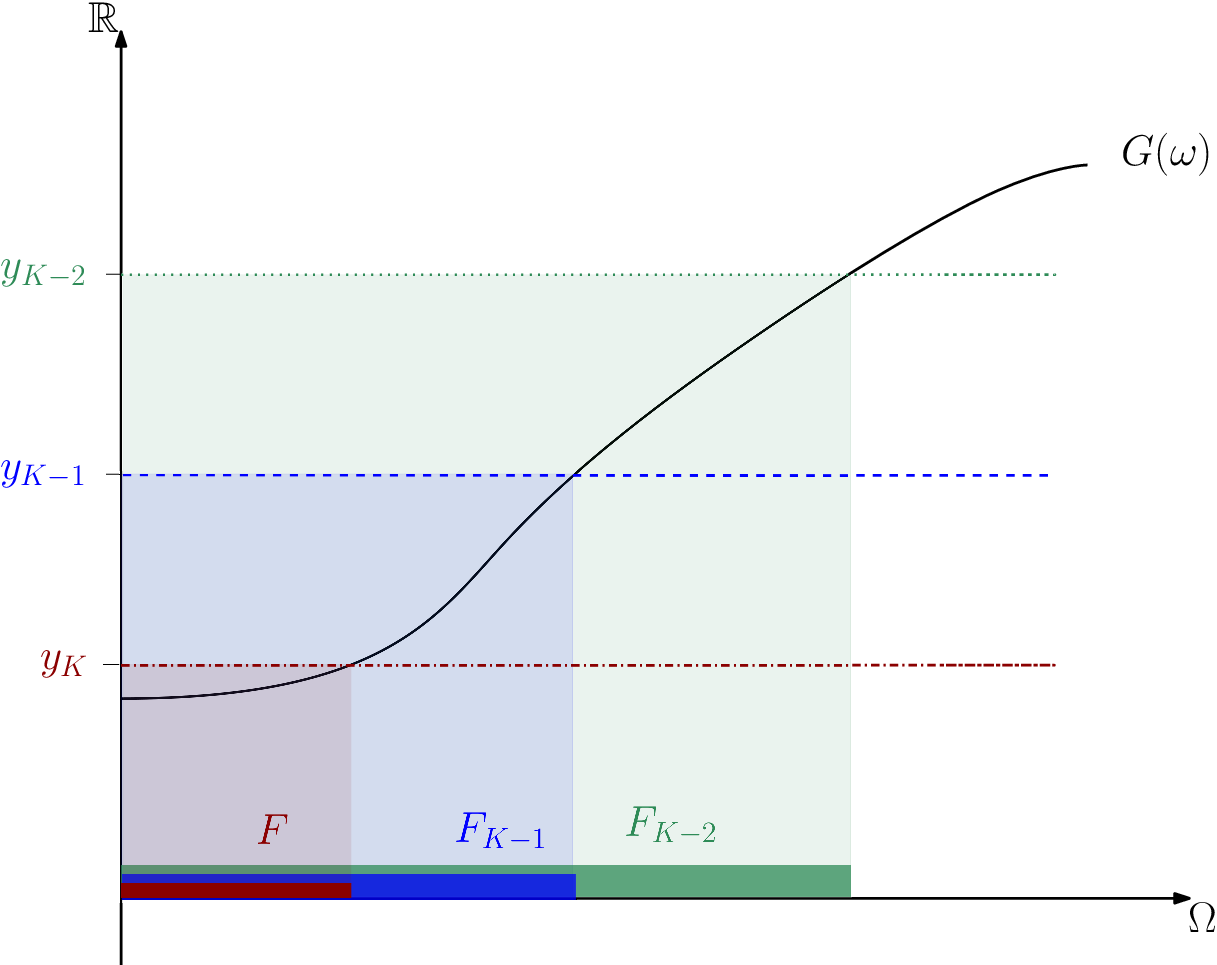}
  \caption{A cartoon of 
  the limit state function $\omega\mapsto G(\omega)$ and the resulting failure sets.
  }
    \label{fig:classical_subset}
\end{figure}

To estimate $\Pr(F)$ from the product of conditional probabilities we need to 
compute
\begin{equation}\label{eq:cond_prob}
\Pr(F_j|F_{j-1}) = \int_{\mathbb R} \mathbbm{1}_{(-\infty, y_j]}(x) \pi_{j-1}(x) \dx, \quad \text{for} \ \ j = 1,\ldots,K,
\end{equation}
where $\pi_{j-1}$ is the pdf of $G$ conditioned on the event $G\le y_{j-1}$, i.e.~
\begin{equation}\label{eq:cond_dens}
	\pi_{j-1}(x) =  \frac{\pi(x)\mathbbm{1}_{(-\infty, y_{j-1}]}(x)}{Z} \quad \text{with} \quad Z := \int_{-\infty}^{y_{j-1}} \pi(x)\,\dx \,.
\end{equation}
The standard MC estimator for the integral in \eqref{eq:cond_prob} is given by
\begin{equation}
   \Pr(F_j|F_{j-1}) \approx \frac{1}{N}\sum_{i=1}^N  \mathbbm{1}_{(-\infty, y_j]}(G^{(i)}),\quad 
G^{(i)} \sim \pi_{j-1}.
\end{equation}
In general, we cannot generate i.i.d.\ samples from $\pi_{j-1}$ directly, at
least not efficiently. In order to circumvent this, MCMC methods are commonly employed, 
see Section~\ref{sec:MCMC} below.

\subsection{Subset Simulation with Selective Refinement}
\label{sec:subset_select}

In practice we also need to take into account the numerical approximations of the failure domains. Instead of having a fixed computational mesh for all
samples, which is the typical approach in the literature, we follow \cite{EHM14} and use a
sample-dependent approximation $G_L(\omega)$ that guarantees 
instead that the error in the limit state function satisfies either the bound \eqref{eq:error_bound} in
Assumption~\ref{ass:error_estimate1} or the weaker bound \eqref{eq:error_bound2} in Assumption~\ref{ass:selective_refinement}. 

Firstly, in the case of Assumption~\ref{ass:error_estimate1}, we use the sequence of intermediate failure sets
\begin{equation}\label{def:fail_sets}
F^L_j:=\{\omega: G_L(\omega)\leq y_j\}, \quad \text{for} \ \
\infty = y_0
> \dots > y_{K-1} > y_K = 0\,,
\end{equation}
which obviously fulfill the critical subset property
\begin{equation}
F^L = F^L_K\subset F^L_{K-1}\subset \dots \subset F^L_1 \subset F_0^L = \Omega.
\end{equation}
In the case of the selective refinement strategy, i.e., under Assumption~\ref{ass:selective_refinement}, the following sequence of intermediate failure sets is chosen:
\begin{equation}\label{def:fail_sets_select}
F^L_j:=\{\omega: G_L^{y_j}(\omega)\leq y_j\}, \quad \text{for} \ \
\infty = y_0
> \dots > y_{K-1} > y_K = 0\,.
\end{equation}
The subset property can be guaranteed again if the failure thresholds are sufficiently far apart.
\begin{lemma}\label{lem:subset0} Let $j\in\{1,\dots,K\}$. Suppose that the random variable $G_L$ satisfies Assumption~\ref{ass:selective_refinement} and that $y_{j-1}-y_j\ge 2\gamma^L$.
Then 
\[
F^L_j=\{\omega: G_L^{y_j}(\omega)\leq y_j\}\subset \{\omega: G_L^{y_{j-1}}(\omega)\leq y_{j-1}\} = F^L_{j-1}.
\]
\end{lemma}
\begin{proof} 
\correction{First note that by convention $F_0 = \Omega$ in which case the subset property is always satisfied. Next, we assume that $j\ge 2$ and}
we will give a proof by contradiction. Fix $\omega\in\Omega$ such that $G_L^{y_j}(\omega)\le y_j$ and $G_L^{y_{j-1}}(\omega) > y_{j-1}$.\medskip

\noindent
\textit{Case 1:} Assume that $|G_L^{y_j}(\omega) -G(\omega)|>\gamma^L$, i.e.~\[\correction{|G_L^{y_j}(\omega) -G(\omega)|\le}\max(|\correction{G_L^{y_{j}}}(\omega) -y_j|,\gamma^L) = y_j - G_L^{y_j}(\omega).\] 
If $|G_L^{y_{j-1}}(\omega) - G(\omega)| \le \gamma^L$ then
\begin{align*}
G_L^{y_{j-1}}(\omega) = G_L^{y_{j-1}}(\omega) - G(\omega) + G(\omega)&\correction{\le \gamma^L + G(\omega)}\\
&\correction{=\gamma^L + G(\omega)- G_L^{y_j}(\omega) +G_L^{y_j}(\omega)}\\ 
&\correction{\le \gamma^L +  |G(\omega)- G_L^{y_j}(\omega)| +G_L^{y_j}(\omega)}\\
&\correction{\le\gamma^L + y_j - G_L^{y_j}(\omega)+G_L^{y_j}(\omega)}\\
&= \gamma^L+y_j \le y_{j-1}, 
\end{align*}
which contradicts $G_L^{y_{j-1}}(\omega) > \correction{y_{j-1}}$. On the other hand, if $|G_L^{y_{j-1}}(\omega) - G(\omega)| > \gamma^L$ then
\[\correction{|G_L^{y_{j-1}}(\omega) - G(\omega)|\le \max(|G_L^{y_{j-1}}(\omega) - y_{j-1}|,\gamma^L)=G_L^{y_{j-1}}(\omega) - y_{j-1} } \]
\correction{and therefore}
\begin{align*}
G_L^{y_{j-1}}(\omega) = G_L^{y_{j-1}}(\omega) - G(\omega) + G(\omega)&\correction{\le |G_L^{y_{j-1}}(\omega) - G(\omega)| + G(\omega)} \\
&\correction{\le G_L^{y_{j-1}}(\omega) - y_{j-1}+G(\omega)}\\
&\correction{= G_L^{y_{j-1}}(\omega) - y_{j-1}+G(\omega)- G_L^{y_j}(\omega) +G_L^{y_j}(\omega)}\\
&\le G_L^{y_{j-1}}(\omega) - y_{j-1} +y_j < G_L^{y_{j-1}}(\omega),
\end{align*}
which is a contradiction in itself.\pagebreak

\noindent
\textit{Case 2:} Now, assume that $|G_L^{y_j}(\omega) -G(\omega)|\le\gamma^L$, i.e.~
\[\correction{|G_L^{y_j}(\omega) -G(\omega)|\le} \max(|\correction{G_L^{y_{j}}}(\omega) -\correction{y_j}|,\gamma^L) = \gamma^L.\] 
Now, if $|G_L^{y_{j-1}}(\omega) - G(\omega)| \le \gamma^L$ then
\[ 
G_L^{y_{j-1}}(\omega) = G_L^{y_{j-1}}(\omega) - G(\omega) + G(\omega)-G_L^{y_j}(\omega)+G_L^{y_j}(\omega) \le 2\gamma^L + y_j \le y_{j-1}, 
\]
which contradicts $G_L^{y_{j-1}}(\omega) > \correction{y_{j-1}}$. If again $|G_L^{y_{j-1}}(\omega) - G(\omega)| > \gamma^L$, then 
\begin{equation*}
\begin{split}
 G_L^{y_{j-1}}(\omega) &= G_L^{y_{j-1}}(\omega) - G(\omega) + G(\omega)-G_L^{y_j}(\omega) +G_L^{y_j}(\omega)\\ &\le G_L^{y_{j-1}}(\omega) - y_{j-1}+ \gamma^L +y_j < G_L^{y_{j-1}}(\omega),
 \end{split}
\end{equation*}
due to the assumed \correction{upper} bound on $y_j - y_{j-1}$, which is again a contradiction in itself. 
\end{proof}

Hence, in both cases the numerical approximation of the rare event probability $P_L = \Pr(F^L)$ on level $L$ can be written as a product of intermediate failure set probabilities as
\begin{equation}\label{eq:subset_est}
  P_L = \prod_{j = 1}^K\Pr(F^L_j \mid F^L_{j-1}).
\end{equation}
Finally, given estimators $\widehat P_j$ for $\Pr\big(F^L_j\mid
  F^L_{j-1}\big)$, we define the subset simulation estimator as
\begin{equation}
  \widehat P^{\mathrm{\mathrm{SuS}}} = \prod_{j=1}^K \widehat P_j.
\end{equation}
In general, this is a biased estimator for $P_L$, but it can be shown that it is asymptotically unbiased~\cite{AB01}. We will return to this in Section \ref{sec:MCMC}. 

\correction{As stated already in Remark \ref{rem:notation}, for simplicity} we will omit the dependence on $y_j$ and write \correction{in the following $G_\ell(\omega)$ instead of $G_\ell^{y_j}(\omega)$, also in the case of selective refinement when using Assumption~\ref{ass:selective_refinement}}.

\subsection{Adaptive Multilevel Subset Simulation}
\label{sec:multilevel}

We will now go one step further and consider the sequence of failure sets 
\[
F^{\text{ML}}_j=\{\omega: G_{\ell_j}(\omega)\leq
y_j\}, \quad j=1,\dots,K,
\]
where each $G_{\ell_j}$ is computed only to tolerance $\gamma^{\ell_j}$ with $\ell_j \le L$, and $y_j$ and $\ell_j$ are adaptively chosen. Typically, we assume $\ell_j \ge \ell_{j-1}$ and $\ell_K = L$. \correction{We consider both the cases without and with selective refinement, i.e., $G_{\ell_j}(\omega)$ or $G^{y_j}_{\ell_j}(\omega)$, respectively.}

Such a multilevel strategy was also at the heart of \cite{UP2015}, but there the thresholds $y_j$ were chosen to roughly balance the contributions from the individual subsets, i.e., $\widehat P_1 \approx \widehat P_2 \approx \ldots \approx \widehat P_K$, as in classical subset simulation. This reduces the number of expensive fine resolution samples only by a linear factor $\mathcal{O}(K)$. Furthermore, without any further conditions the crucial subset property $F^{\text{ML}}_{j}\subset F^{\text{ML}}_{j-1}$ cannot be guaranteed. Thus, one focus of \cite{UP2015} was to estimate the correction factor that arises due to the loss of the subset property. 

To circumvent this problem, in the following lemma we propose
a method whereby the thresholds $y_j$ are chosen adaptively, to maintain the subset property. Crucially, we exploit here the sample-wise error bounds in Assumptions~\ref{ass:error_estimate1} and \ref{ass:selective_refinement}.

\begin{lemma}
  \label{lem:subset}
  Consider the sequence $F^{\text{ML}}_j,\ j=1,\dots,K,$ 
  where either Assumption~\ref{ass:error_estimate1} or Assumption~\ref{ass:selective_refinement} is satisfied. Let 
  $y_K = 0$ and $y_0 = \infty$, and choose 
  \[
  y_{j}=
  y_{j+1} + (\gamma^{\ell_{j}} + \gamma^{\ell_{j+1}}) > 0, \quad \text{for }\ j=K-1,\ldots,1. 
  \]
  Then, the subset property
  $F^{\text{ML}}_{j+1}\subset F^{\text{ML}}_{j}$ holds. 
\end{lemma}
\begin{proof}
First suppose that Assumption~\ref{ass:error_estimate1} holds. Then 
\begin{equation}
    |G_{\ell_j}(\omega)-G_{\ell_{j+1}}(\omega)|\leq |G(\omega)-G_{\ell_j}(\omega)| + 
|G(\omega)-G_{\ell_{j+1}}(\omega)|\leq (\gamma^{\ell_j}+\gamma^{\ell_{j+1}}),
  \end{equation}
  and hence, if $G_{\ell_{j+1}}(\omega)\leq y_{j+1}$ then 
  \begin{equation}
    G_{\ell_j}(\omega)\leq G_{\ell_{j+1}}(\omega) + (\gamma^{\ell_j}+\gamma^{\ell_{j+1}})\leq y_{j+1}+ 
(\gamma^{\ell_j}+\gamma^{\ell_{j+1}})=y_j,
  \end{equation}
  which concludes the proof for Assumption~\ref{ass:error_estimate1}.
  
  The proof for Assumption~\ref{ass:selective_refinement} then follows directly since $y_j-y_{j-1}\ge 2\gamma^{\ell_j}$.
\end{proof}

For simplicity, we assume that $K = L$ and that $\ell_j = j$ for all $j=1,\dots,K$. The rare event probability can then be written as
\begin{equation}\label{eq:MLdecomp}
  \Pr(F^{\text{ML}}_L) = \prod_{\ell = 1}^L\Pr(F^{\text{ML}}_{\ell}\mid F^{\text{ML}}_{\ell-1}),
\end{equation}
where $F^{\text{ML}}_\ell = \{ \omega: G_{\ell}(\omega)\leq y_\ell\}$. However, this does not preclude us from using more than $L$ subsets. If the first intermediate failure set $F^{\text{ML}}_1$ is still a rare event, we can estimate $\Pr(F^{\text{ML}}_{1} | F^{\text{ML}}_{0}) = \Pr(F^{\text{ML}}_{1})$ using classical subset simulation with an additional $K_1$ subsets instead of plain MC, but with all evaluations of the limit state function on those additional $K_1$ subsets only computed to an accuracy of $\gamma^1$ in Assumptions~\ref{ass:error_estimate1} and \ref{ass:selective_refinement}. \update{This may also be necessary for intermediate failure probabilities $P(F_j | F_{j-1})$ if they happen to be very small.}

Let $\widehat P_1$ be an estimator for $\Pr(F^{\text{ML}}_1)=\Pr(F^{\text{ML}}_1|F^{\text{ML}}_0)$ using standard MC or classical
subset simulation on discretisation level $\ell=1$. For $\ell>1$ we assume as above that we are given estimators
$\widehat P_\ell$ for $\Pr\left(F^{\text{ML}}_\ell\mid F^{\text{ML}}_{\ell-1}\right)$ that will be constructed by 
MCMC sampling in the following. 
We define the multilevel subset simulation estimator by
\vspace{-0.25cm}
\begin{equation}\label{eq:ML_estimator}
 \widehat 
P^{\mathrm{ML}} = \widehat P_1\prod_{\ell=2}^L \widehat P_\ell.\vspace{-0.25cm}
\end{equation}
In practice, we propose to apply classical subset simulation for the estimation of $\widehat P_1$. 

{An illustration of the idea behind Lemma \ref{lem:subset} is given in Figure
\ref{fig:subset_new}. An additional benefit of the choice of thresholds in Lemma \ref{lem:subset} is that 
\correction{the difference in the intermediate failure thresholds $y_\ell$} shrinks geometrically as $\ell$ increases. See Figure~\ref{fig:subset} for an illustration. Thus, \correction{by continuity of $\mathbb P$ it follows that $\mathbb P(F_\ell)\to \mathbb P(F)$, and therefore also} $\Pr(F^{\text{ML}}_{\ell}\mid F^{\text{ML}}_{\ell-1}) \to 1$ as $\ell \to \infty$. \correction{This effect} 
reduces the variance on the latter subsets that have to be computed to higher accuracy. As a consequence, significantly fewer samples have to be computed at high accuracy reducing the overall complexity of the estimator dramatically. We will return to this point in Section \ref{sec:complexity}.}

\begin{figure}[t]
\centering \includegraphics[width=0.49\textwidth]{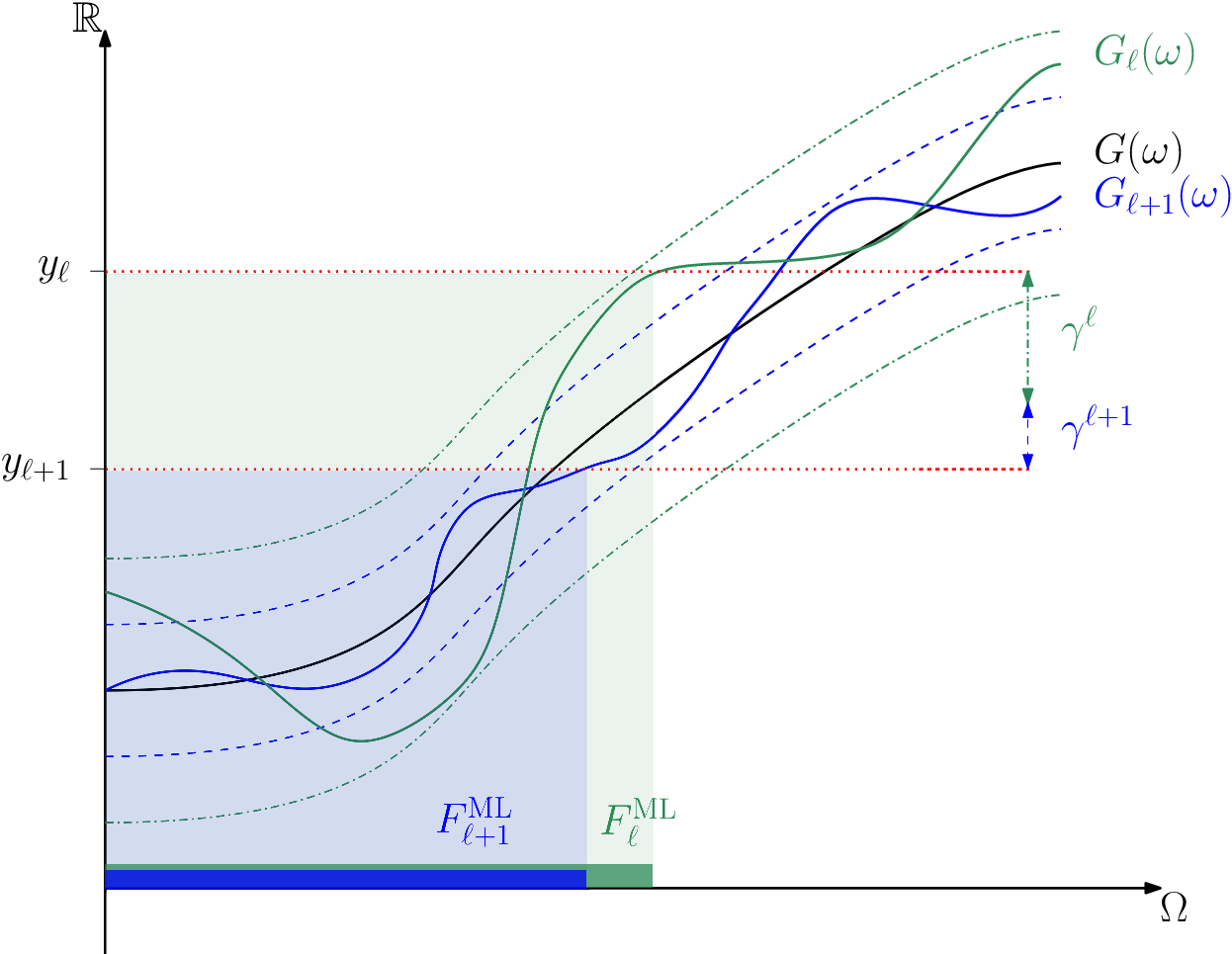}~~\includegraphics[width=0.49\textwidth]{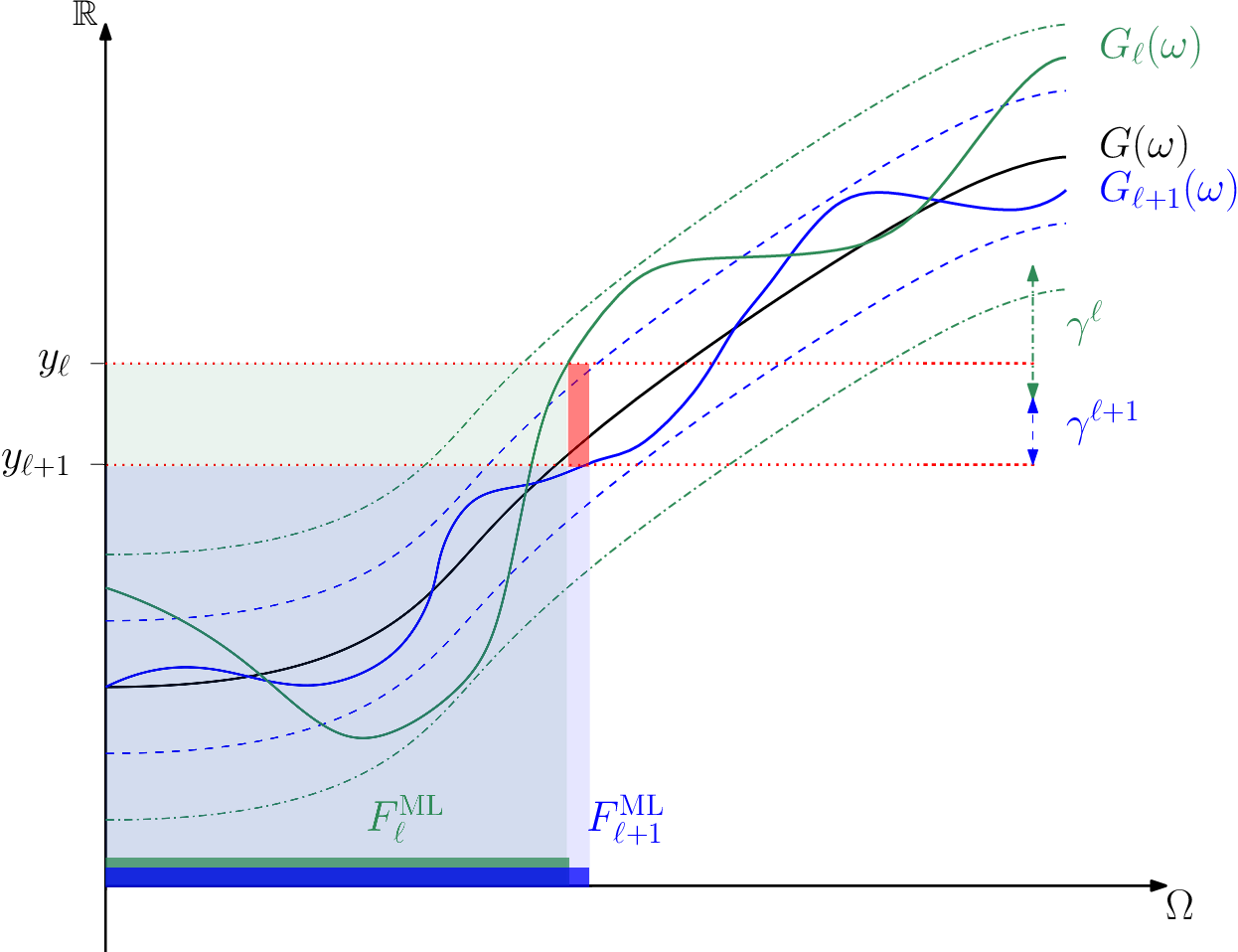}
\caption{Illustration of the choice of failure thresholds in Lemma~\ref{lem:subset}, which guarantee the subset property 
$F^{\text{ML}}_{\ell+1}\subset F^{\text{ML}}_\ell$. It holds,
when $y_{\ell}-y_{\ell+1}=\gamma^{\ell+1}+\gamma^{\ell}$ and $|G_{\ell}(\omega)-G_{\ell+1}(\omega)|\leq 
  \gamma^{\ell+1}+\gamma^{\ell}$ as chosen in Lemma \ref{lem:subset} (left plot). If $y_{\ell}-y_{\ell+1}<\gamma^{\ell+1}+\gamma^{\ell}$, then Lemma~\ref{lem:subset} does not apply. In particular, for too small choices of $y_{\ell}-y_{\ell+1}$ we may find $\omega\in\Omega$ such that $G_{\ell+1}(\omega)\le y_{\ell+1}$, but $G_{\ell}(\omega)> y_{\ell}$, as illustrated by the red zone (in the right plot).}\label{fig:subset_new}
\end{figure} 

\begin{figure}[t]
  \centering
  \tikzstyle{int}=[draw, fill=blue!20, minimum size=2em]
  \tikzstyle{init} = [pin edge={to-,thin,white}]
  \begin{tikzpicture}\small
    \draw[-,semithick] (0.5,0) -- (1,0);
    \draw[|-,semithick] (1,0) -- (2.5,0);
    \draw[|-,semithick] (2.5,0) -- (5.5,0);
    \draw[|-,semithick] (5.5,0) -- (10.5,0);
    \draw[|->,semithick] (10.5,0) -- (11,0);
    \draw[<->,semithick] (1,0.2) -- (2.5,0.2);
    \draw[<->,semithick] (2.5,0.2) -- (5.5,0.2);
    \draw[<->,semithick] (5.5,0.2) -- (10.5,0.2);
    \draw (1,-0.5) node {$y_{\ell+3}$};
    \draw (2.5,-0.5) node {$y_{\ell+2}$};
    \draw (5.5,-0.5) node {$y_{\ell+1}$};
    \draw (10.5,-0.5) node {$y_{\ell}$};
    \draw (1.9,0.6) node {$\gamma^{\ell+3}+\gamma^{\ell+2}$};
    \draw (4.2,0.6) node {$\gamma^{\ell+2}+\gamma^{\ell+1}$};
    \draw (8,0.6) node {$\gamma^{\ell+1}+\gamma^{\ell}$};
  \end{tikzpicture}
  \caption{Illustration 
  of the shrinkage of the subsets $F^{\text{ML}}_{\ell}$ with increasing $\ell$. 
  }
  \label{fig:subset}
\end{figure}

\update{The three considered algorithms, namely subset simulation, subset simulation with selective refinement and multilevel subset simulation,
are summarised in Appendix~\ref{app:algorithm}. To estimate the conditional probabilities in \eqref{eq:subset_est} and \eqref{eq:MLdecomp} a MCMC algorithm is used, which will be described in the following section. The different types of subset simulation only differ in the construction of the considered subsets $F_j$.
We have further included an adaptive choice for $N_\ell$ in Algorithm~\ref{alg:MLsubset} in order to verify the required tolerance with respect to the rRMSE~$\delta(\widehat P)$. Sufficient bounds on $\delta(\widehat P)$ are derived in Section \ref{sec:complexity_ss_select}.
A similar choice can also be included in Algorithm~\ref{alg:classicalsus}.}

\section{Shaking Transformations and Asymptotic Convergence}\label{sec:MCMC}

From the above discussion we see that a crucial component of any subset simulation is an efficient estimator for the conditional
probabilities $\Pr(F_j|F_{j-1})$, for $j>1$. In other words, we need to generate samples from
the pdf given in~\eqref{eq:cond_dens}.
Due to the small probability of failure, standard MC sampling is infeasible. Instead, several authors have developed MCMC algorithms \cite{AB01,PaWoKiZwSt14,AuPa16}. 
Here,  we use the general, parallel, one-path (POP) algorithm based on shaking transformations introduced in \cite{GL2015} and further developed in \cite{ADGL2017}.  The key idea is to build one Markov chain,  exploring the space through shaking transformations and moving from the coarsest subset to the finest via a sequence of rejection operators. 

The shaking transformation of a random variable $X$ with respect to a random variable~$\correction{W}$, acting on measurable spaces $(\mathbb X, \mathcal X)$ and \correction{$(\mathbb W,\mathcal W)$}, respectively, can be defined as a measurable mapping $S:\mathbb X\times \correction{\mathbb W}\to \mathbb X$ that satisfies 
\begin{equation}\label{eq:shakingtransf}
(X,S(X,\correction{W})) = (S(X,\correction{W}),X)
\end{equation}
in distribution.  Assuming that the rare event and the corresponding intermediate subsets can be written as events of $X$, i.e.,
\[F = \{X\in A\},\quad F_j= \{X\in A_j\},\quad \text{such that} \quad A,A_j\in\mathcal X,\quad j=1,\dots,K,\]
the shaking transformation will act as the proposal in the constructed MCMC algorithm. 
In order to force the Markov chain to explore the subset $F_j$,  we define the rejection operator
\[M_j^S: \mathbb X\times \correction{\mathbb W}\to \mathbb X,\quad \text{with}\quad M_j^S(x,\correction{w}) = S(x,\correction{w})\mathbbm{1}_{S(x,\correction{w})\in A_j}+x\mathbbm{1}_{S(x,\correction{w})\notin A_j}.\]

The proposed MCMC algorithm in the earlier multilevel subset simulation paper \cite{UP2015} is based on a preconditioned Crank-Nicholson (pCN) \cite{Neal98,CRSW13} proposal which can in fact be written as a shaking transformation and hence the theoretical results in \cite{GL2015} can be applied. The property of being a shaking transformation is related to the detailed balance condition. This observation has been made and verified already in \cite[Thm.~8]{ADGL2017}. The general POP algorithm is given in Algorithm~\ref{alg:POP}.

\begin{algorithm}[t]
\caption{Parallel One-Path algorithm [POP].}
\label{alg:POP}
\begin{algorithmic}[1]
\State Given a seed $X_{0,0}\sim X$
\For{$j=0,\dots,K-1$} 
\For{ $i=0,\dots,N_{j+1}-1$}
\State Generate \correction{$W_{j,i}\sim W$},
\State Shake and accept/reject $X_{j,i+1} = M_j^S(X_{j,i},\correction{W_{j,i}})$
\EndFor
\State Estimate the probability of subset $F_{j+1}$ by\vspace{-1ex} 
\[
\qquad\qquad\qquad
\widehat P_{j+1} = \frac1{N_{j+1}}\sum\limits_{k=0}^{N_{j+1}-1} \mathbbm{1}_{A_{j+1}}(X_{j,k}),\vspace{-1ex} 
\]
\State $i_j = \arg\min\{k\mid \correction{X_{j,k}} \in A_{j+1}\}$
\State Define initial state for next level $X_{j+1,0} = X_{j,i_j}$
\EndFor
\State \textbf{Result:} $\widehat P = \prod_{j=1}^{K} \widehat P_j$
\end{algorithmic}
\end{algorithm}

Assuming that the resulting Markov chains $(X_{j,i})_{i\ge0}$ are $\pi_j$--irreducible and Harris recurrent under a small set condition, the POP algorithm converges in the sense that for every $j$ there exists a constant $C_j>0$ such that
\[
\mathbb E\left[ |\widehat P_j - \mathbb P(X\in A_{j+1}\mid X\in A_j)|^2\right] \le \frac{C_j}{N_j}.
\]

\begin{algorithm}[t]
\caption{Parallel One-Path algorithm based on Gaussian transformation.}
\label{alg:POP2}
\begin{algorithmic}[1]
\State Given a seed $\Theta_{0,0}\sim \Theta$ and a correlation parameter $\eta\in[0,1]$
\For{$j=0,\dots,K-1$} 
\For{ $i=0,\dots,N_{j+1}-1$}
\State Generate $\correction{W_{j,i}}\sim \Theta$,
\State Shake and accept/reject $\Theta_{j,i+1} = M_j^{S_\eta}(\Theta_{j,i},\correction{W_{j,i}})$
\EndFor
\State Estimate the probability of subsets $F_{j+1}$ by\vspace{-1ex} 
\update{\[
\qquad\qquad\qquad 
\widehat P_{j+1} = \frac1{N_{j+1}}\sum\limits_{k=0}^{N_{j+1}-1} \mathbbm{1}_{A_{j+1}}(\Theta_{j,k}),\vspace{-1ex}
\]}
\State \update{$i_j = \arg\min\{k\mid \Theta_{j,k} \in A_{j+1}\}$}
\State Define initial state for next level $\Theta_{j+1,0} = \Theta_{j,i_j}$
\EndFor
\State \textbf{Result:} $\widehat P = \prod_{j=1}^{K} \widehat P_j$
\end{algorithmic}
\end{algorithm}

In \cite{ADGL2017} the authors include a model for $X$ via Gaussian transformations which has also been considered in \cite{PaWoKiZwSt14}.  We consider 
$X$ modelled through a 
\update{measurable} transformation $T$ as
\[ X(\omega) = T(\Theta(\omega)), \ \omega\in\Omega,\ T:\mathbb R^d \to \mathbb X\ \text{measurable},\]
where $\Theta\sim\mathcal N(0,I_{\mathbb R^d})$.
The underlying rare event can then be formulated as 
\[ 
\update{A = \{\theta\in\mathbb R^d \mid \varphi(\theta,\bar y)\le 0\} \subset \mathcal B(\mathbb R^d),\quad \varphi:\mathbb R^d\times \mathbb R\cup \{\infty\}\to \mathbb R}
\]
\update{where $\varphi$ is non--increasing in the second component in the sense that $\varphi(\theta,y)\ge \varphi(\theta,y')$, for any $\theta\in\mathbb R^d$ and $y\ge y'$, and the convention $\varphi(\theta,\infty):=-\infty$ is assumed.  Further, we assume that $\varphi$ is measurable in the first component. The nested subsets are built through a sequence of level parameters $\bar y:= y_K<\dots<y_k<\dots<y_0\le\infty$ by
\begin{equation}\label{eq:subsets_GT}
A_j = \{\theta\in\mathbb R^d\mid \varphi(\theta,y_j)\le 0\} \in \mathcal B(\mathbb R^d),
\end{equation}
such that 
\[
A:= A_K\subset \dots \subset A_j\subset\dots \subset A_0 := \mathbb R^d.
\]
In fact, the special case $\varphi(\theta,y):= T(\theta)-y$ for a measurable model functional $T:\mathbb R^d\to\mathbb R$
connects the present approach to the setting considered in the previous sections. The sequence of subsets $F_j$ considered above can then be written as $F_j=\{\update{\Theta}\in A_j\} \in\Sigma$.}

Consider the shaking transformation 
\begin{equation}\label{eq:pcn_shaking}
    S_\eta(\update{\theta},\correction{w}) = \sqrt{1-\eta^2}\update{\theta}+\eta \correction{w},
\end{equation} 
which obviously satisfies \eqref{eq:shakingtransf}, and let \correction{$W$} be an independent copy of $\Theta$.  The shaking transformation (with rejection) is now defined as 
\update{\begin{equation}\label{eq:rejection_operator}
M_j^{S_\eta}:\begin{cases}\mathbb R^d\times\mathbb R^d\to \mathbb R^d,\\ (\update{\theta},\correction{w})\mapsto S_\eta(\update{\theta},\correction{w})\mathbbm{1}_{S_\eta(\update{\theta},\correction{w})\in A_j} + x \mathbbm{1}_{S_\eta(\update{\theta},\correction{w})\notin A_j} \end{cases}.
\end{equation}}
We note that this is in fact equivalent to using the pCN approach, which is in detailed balance with the prior. This has the advantage that the acceptance rate does not
decrease with the dimension, which is the case, e.g., for traditional random walk proposals. The POP algorithm is then of very similar form as one of the algorithms proposed in \cite{PaWoKiZwSt14} to generate samples from the conditional pdf $\pi_{j-1}$. For completeness, details are provided in Algorithm \ref{alg:POP2}. 
Under certain assumptions 
geometric convergence of the resulting estimator can be verified. 

\begin{theorem}[{\cite[Thm.\ 8]{ADGL2017}}]\label{thm:ergodicity}
Let $j\in\{0,\dots,K-1\}$ be fixed and consider the Markov chain $(\Theta_{j,i})_{i\ge0}$ resulting from 
Algorithm~\ref{alg:POP2}. Assume that 
\begin{equation}\label{eq:suff_small_set}
\update{\sup_{\theta\in A_j} \mathbb P(S(\theta,W)\notin A_j)}=\delta_1<1
\end{equation}
and let $c\ge2$ and $q\ge 2$. Define $V(\theta) = \exp(c \sum_{k=1}^d\theta_k)$ and assume that the 
initial condition $\Theta_{j,0}$ is independent of the future evolution of 
the Markov chain and such that $\mathbb E \left[V(\Theta_{j,0})\right] <+\infty$. Then, there exists $C>0$ and a geometric rate $r\in(0,1)$ such that for any measurable function $g:\mathbb R^d\to \mathbb R$ with 
\[\sup_{\theta\in\mathbb R^d} \frac{|g(\theta)|}{V^{1/q}(\theta)}<+\infty\] it holds true that
\begin{align*}
\update{\mathbb E\left[\left|\frac1N \sum_{i=1}^N g(\Theta_{j,i}) - \mathbb E[g(\Theta) \;\big| \; \Theta\in A_j]\right|^q\right]} &\le C N_j^{-q/2},\\[0.5ex]
\text{and} \qquad \update{\left|\mathbb E[g(\Theta_{j,N}] -\mathbb E[g(\Theta)\;\big|\; \Theta\in A_j]\right|} &\le C r^n.
\end{align*}
\end{theorem}

In this setting, the estimator $\widehat P_j$ of $\Pr(F_j|F_{j-1})$ is then given by
\begin{equation}\label{eq:estim_comnd}
	\update{\widehat P_j = \frac{1}{N}\sum_{i=1}^N\mathbbm{1}_{\{\Theta_{j-1,i}\in A_j\}} = \frac{1}{N}\sum_{i=1}^N\mathbbm{1}_{(-\infty,0]}(\varphi(\Theta_{j-1,i}\,,y_j))}.
\end{equation}
\update{Note that by using a measurable transformation $T_\ell$ such that $G_\ell(\omega) = T_\ell(\Theta(\omega))$, the limit state function can also be defined pointwise in $y_\ell$, i.e., $\varphi(\theta,y_\ell):= T_\ell(\theta)-y_\ell$. In that case, the subset property follows from Lemma~\ref{lem:subset} and Theorem~\ref{thm:ergodicity} can be applied to the adaptive multilevel subset simulation algorithm as well.}

\correction{A brief discussion of the assumptions is in order. To ensure condition \eqref{eq:suff_small_set} the acceptance rate of the algorithm needs to be uniformly bounded from below. If
\begin{equation*}
\update{\sup_{\theta\in A_j} \mathbb P(S(\theta,W)\notin A_j)}=1,
\end{equation*} 
the Markov chain may get stuck in some points. The condition is problem dependent and in general difficult to verify. However, this is a common problem with MCMC algorithms.}

Since the sequences of samples $\{\Theta_{j-1,i}\}_{i=1}^N$ are dependent, we use convergence diagnostics to estimate the autocorrelation
factor in the sequence. A common measure of the
dependence in the sequence is the autocorrelation factor. Following \cite{AB01}, we use multiple chains
to compute an estimate of the autocorrelation factor $\phi_{j}$. The total number of samples~$N$ to achieve a certain accuracy of the estimator in \eqref{eq:estim_comnd} is \correction{again determined by the rRMSE in \eqref{eq:def_cov}:}
\begin{equation*}
\delta\big(\widehat P_j\big) =
\frac1{\Pr(F_j|F_{j-1})} \sqrt{\correction{\mathbb E}\left[\big(\widehat P_j - \Pr(F_j|F_{j-1})\big)^2\right]}\, .
\end{equation*}
Expanding the square of the rRMSE, 
we can estimate 
\begin{equation}\label{eq:cov_subset}
\delta(\widehat P_j)^2 = \correction{\frac{1 - \mathbb E\left[\widehat P_j\right]}{N \Pr(F_j|F_{j-1})} = \frac{1-
    \Pr(F_j|F_{j-1})}{N\Pr(F_j|F_{j-1})}(1+\phi_j),}
\end{equation}
with autocorrelation factor $\phi_j=0$ 
for i.i.d.~samples and $\phi_j>0$ for dependent samples. \correction{We note that for each $j$ the expectation in \eqref{eq:cov_subset} is with respect to 
the joint distribution of the $\Theta_{j-1,i},\ i=1,\dots,N$.} 

\medskip

\section{Complexity Analysis} 
\label{sec:complexity}

We will now bound the complexity of the new estimators.

\subsection{Complexity of subset simulation with selective refinement}
\label{sec:complexity_ss_select}

{\correction{In analogy to~\eqref{expand_del}, the square $\delta\big(\widehat P^{\text{SuS}} \big)^2$ of the rRMSE for the subset simulation estimator can be expanded into the sum of a bias term -- identical to the one in \eqref{expand_del} for standard MC -- and the relative variance $\mathbb V [\widehat
  P^{\mathrm{SuS}}]/P^2_L$\,.}} The bias term can again be made less than
$\mathrm{TOL}^2/2$ by choosing $L$ as before in \eqref{def:L}.
Moreover, applying results derived in \cite{AB01} it turns out that
\begin{equation}\label{eq:var_bound_sus}
\mathbb E \bigg[\frac{(\widehat  P^{\mathrm{SuS}}-P_L)^2}{P_L^2} \bigg]\leq \frac{\text{TOL}^2}{2}
\quad \text{if} \quad \delta(\widehat P_j)^2 = \correction{\mathbb E} \bigg[\frac{(\widehat P_j-P_j)^2}{P_j^2}  \bigg] \propto K^{-s} \,\text{TOL}^2\,,
\end{equation}
where $s=1$ if the estimators $\widehat
P_j$, $j=1,\ldots,K$, are uncorrelated and $s=2$ if they are fully correlated. To simplify notation, the biased densities defined by $F_{j-1}^L$ instead of $F_{j-1}$ in \eqref{eq:cond_dens} are again denoted by 
$\pi_{j-1}$ suppressing the dependence on $L$.  

Let us now estimate the cost for the particular case where $\widehat
P_j $ is chosen to be the MCMC estimator for $\Pr\big(F_j\mid
  F_{j-1}\big)$ described in
Section~\ref{sec:MCMC} with $N_{j}$ samples.  As it has been shown in \cite{AB01} the \correction{rRMSE} can be computed as
\begin{equation}\label{eq:error_level1}
\delta(\widehat P_j)^2 = \frac{1-\Pr\left(F_j\mid F_{j-1}\right)}{N_j \Pr\left(F_j\mid F_{j-1}\right)}(1+\phi_j)
\end{equation}
where $\phi_j$ is the autocorrelation factor of the Markov
chain produced by Algorithm~\ref{alg:POP2}.  We note that \eqref{eq:error_level1} crucially depends on the fact that the underlying Markov chain generated through the MCMC algorithm is ergodic. The ergodicity of the Markov chain can be obtained from Theorem~\ref{thm:ergodicity}.  

In the following complexity analysis, we assume that \eqref{eq:error_level1} is satisfied. Moreover, we assume an adaptive selection of the
failure sets has been applied,  e.g.~as described in \cite{AB01}, where the values of $y_j$
are chosen in the course of the algorithm, such that $\widehat P_j \approx p_0$, for all
$j=1,\ldots,K$. It was shown in \cite{ZuBe12} that the best
performance is achieved for values of the constant $p_0 \in
[0.1,0.3]$.  For simplicity, we assume (without loss of generality) that $\Pr(F_j\mid F_{j-1})=p_0$ for all $j=1,\ldots,K$. \update{However, since the choice of $p_0$ is somewhat arbitrary, it could also be chosen level-dependent so that the condition $y_{j-1}-y_j\ge 2\gamma^L$ is satisfied on all subsets.}

\begin{theorem}\label{thm:complexity_sus} 
Suppose that 
\begin{enumerate}
\item Assumption~\ref{ass:error_estimate1} is satisfied for some $q \ge 0$ and $\gamma \in (0,1)$,
\item the subsets are chosen such that 
$\Pr(F_j\mid F_{j-1}) = p_0 := P_L^{1/K}$, for all $j=1,\ldots,K-1$, and
\item there exists a constant $\phi < \infty$ such that
$\phi_j\le \phi$ {in \eqref{eq:error_level1}}, for all $j=1,\ldots,K$.
\end{enumerate}
Then, for any $\text{TOL}>0$ the maximum discretisation level $L$ and the numbers of samples $\{N_j\}$ can be chosen such that
\[
\mathbb E \bigg[\frac{|\widehat P^{\mathrm{SuS}}- P|^2}{P^2}\bigg]\le \text{TOL}^2
\]
with a cost that is bounded by
\[
\mathbb{E} \left[\costb{\widehat P^{\mathrm{SuS}}} \right]\le c_4\frac{P_L^2}{P^2}  \text{TOL}^{-(2+q)} \Big(\log\big(P_L^{-1}\big)\Big)^{s+1}
\]
for some constant $c_4>0$. Here, $s=1$ if the estimators $\widehat P_j$, $j=1,\ldots,K$, are uncorrelated and $s=2$ if they are fully correlated.

If in addition Assumption~\ref{ass:selective_refinement} \update{holds and the conditions $y_{j-1}-y_j\ge 2\gamma^L$ are satisfied for all $j=1,\ldots K$, then} the cost can be bounded by 
\[
\mathbb{E} \left[\costb{\widehat P^{\mathrm{SuS}}}\right] \le c_4\frac{P_L^2}{P^2}  \text{TOL}^{-\max(2,1+q)} \Big(\log\big(P_L^{-1}\big)\Big)^{s+1}.
\]
\end{theorem}
\begin{proof}
We first split 
\begin{equation}\label{eq:sus_proof1}
\mathbb E\bigg[\frac{|\widehat P^{\mathrm{SuS}}- P|^2}{P^2}\bigg] = \bigg(\frac{\mathbb E[\widehat P^{\mathrm{SuS}}-P]}{P}\bigg)^2 + \frac{P_L^2}{P^2}\mathbb E\bigg[\frac{(\widehat P^{\mathrm{SuS}}-P_L)^2}{P_L^2}\bigg],
\end{equation}
where the first term can be bounded by
\begin{equation}\label{eq:sus_proof2}
\bigg(\frac{\mathbb E[\widehat P^{\mathrm{SuS}}-P]}{P}\bigg)^2\le 2 \bigg(\frac{P_L-P}{P}\bigg)^2 + 2\frac{P_L^2}{P^2}\bigg(\frac{\mathbb E[\widehat P^{\mathrm{SuS}}-P_L]}{P_L}\bigg)^2.
\end{equation}
To bound the bias error we choose $L\propto\log(\text{TOL}^{-1})$
as in \eqref{def:L}. Due to the assumptions of the theorem we can choose $N_j = N$ uniformly across all subsets. 
The expressions
\[A:=\bigg(\frac{\mathbb E[\widehat P^{\mathrm{SuS}}-P_L]}{P_L}\bigg)^2 \quad \text{and}\quad B:=\mathbb E\bigg[\frac{(\widehat P^{\mathrm{SuS}}-P_L)^2}{P_L^2}\bigg]\]
have been analysed in \cite{AB01}, where it has been shown that
\[A\le \sum_{i,j=1,\ j>i}^K \delta(\widehat P_i)\delta(\widehat P_j) + o\left(\frac1{N}\right)\quad\text{and}\quad B\le \sum_{i,j=1}^K \delta(\widehat P_i)\delta(\widehat P_j) + o\left(\frac1{N}\right)\]
if the estimators $\{\widehat P_j\}$ are correlated, whereas 
\[A\le \sum_{j=1}^K \delta(\widehat P_j)^2\quad\text{and}\quad B\le \sum_{j=1}^K \delta(\widehat P_j)^2\]
if the estimators $\{\widehat P_j\}$ are uncorrelated. 
If $N$ is now chosen such that
\begin{equation}\label{eq:bound_cov}
\delta(\widehat P_j)^2\le c_3K^{-s}\frac{\text{TOL}^2}{5}\frac{P^2}{P_L^2}
\end{equation}
it follows that
\begin{equation}\label{eq:sus_proof3}
A\le c_3\frac{\text{TOL}^2}{5}\frac{P^2}{P_L^2}\quad \text{and}\quad  B\le c_3\frac{\text{TOL}^2}{5}\frac{P^2}{P_L^2}\,.
\end{equation}
Taking into account the choice of $L$ in \eqref{def:L} and combining \eqref{eq:sus_proof1}, \eqref{eq:sus_proof2} and \eqref{eq:sus_proof3}, finally leads to 
\[
\mathbb E\bigg[\frac{|\widehat P^{\mathrm{SuS}}- P|^2}{P^2}\bigg] \le \text{TOL}^2\,.
\]
Now, using the third assumption of the theorem, i.e., that $\phi_j\leq\phi$, a sufficient
condition on the 
number of samples to ensure that the bound in \eqref{eq:bound_cov}
holds is
\begin{equation*}
N 
\propto \frac{P_L^2}{P^2} K^{s} \text{TOL}^{-2} \,,
\end{equation*}
where the proportionality constant depends on $p_0$, $\phi$ and $c_3$,
but is independent of $\text{TOL}$ and $K$. 
Then, recalling that 
$P_L = p_0^K$ and thus $K  \propto \log\big(P_L^{-1}\big)$ it follows that
\[
N \propto \Big(\log\big(P_L^{-1}\big)\Big)^{s+1} \frac{P_L^2}{P^2} \text{TOL}^{-2}\,.
\]
Hence, combining this with the assumed bound on the cost per sample in Assumption~\ref{ass:error_estimate1} and using the fact that $L\propto\log(\text{TOL}^{-1})$ implies $\gamma^{-qL} \propto \text{TOL}^{q}$, the total cost to compute the subset
simulation estimator can be bounded by 
\begin{equation*}
  \mathbb E\left[\costb{\widehat P^{\mathrm{SuS}}}\right] \le KN c_0 \gamma^{-Lq} \le c_4\frac{P_L^2}{P^2}  \text{TOL}^{-(2+q)} \Big(\log\big(P_L^{-1}\big)\Big)^{s+1} \,, 
\end{equation*}
for some constant $c_4>0$. 
Under the selective refinement Assumption~\ref{ass:selective_refinement} and for $q>1$, the computational cost can be bounded by 
\[ 
  \mathbb E\left[\costb{\widehat P^{\mathrm{SuS}}}\right]  \le KN c_0 (1+\gamma^{(1-q)L}) \le c_4\frac{P_L^2}{P^2}  \text{TOL}^{-\max(2,1+q)} \Big(\log\big(P_L^{-1}\big)\Big)^{s+1} \,.
\]
\end{proof}
\begin{remark}
    \update{We note that following Lemma~\ref{lem:subset0}, the assumption $y_{j-1}-y_j\ge 2\gamma^L$, for $j=1,\dots,K,$ is needed to ensure the subset property in the case of
    Assumption~\ref{ass:selective_refinement} (i.e.~for subsets defined through $F_j = \{\omega\in\Omega:\ G_L^{y_j}(\omega)\le y_j\}$). Although this condition may fail in practical scenarios, it is easy to check on the fly and then to apply the selective refinement strategy only on subsets where it is satisfied. The presented complexity analysis in Theorem~\ref{thm:complexity_sus} can then be viewed as the best-case scenario. In many cases it will apply. In the worst case, the complexity is the same as the cost of classical subset simulation.}
\end{remark}

\subsection{Complexity of adaptive multilevel subset simulation}

We now turn our attention to the complexity of the adaptive multilevel subset estimator defined in \eqref{eq:ML_estimator}.
We will make use of the following convergence property in the analysis below, assuming that we can control the ratio of the subset probabilities sufficiently well.
\begin{lemma}\label{lem:level_bound}
Suppose Assumption~\ref{ass:error_estimate1} is satisfied and $y_\ell$ is defined as in Lemma~\ref{lem:subset}. 
Furthermore, we assume that\vspace{-2ex}
\begin{equation}\label{eq:subset_ratio}
\qquad\qquad\frac{\Pr(G\le y)}{\Pr(G\le x)} \le 1+\tilde c|y-x|\log(P^{-1}),
\end{equation}
for $y>x>0$ and $\tilde c>0$ independent of $P$. 
Then the difference between the intermediate probabilities can be bounded in the following way
\begin{equation*}
{\frac{1-\Pr(F^{\text{ML}}_\ell\mid F^{\text{ML}}_{\ell-1} )}{\Pr(F^{\text{ML}}_\ell\mid F^{\text{ML}}_{\ell-1})}} = \frac{1-\Pr(G_\ell\le y_l\mid G_{\ell-1} \le y_{\ell-1})}{\Pr(G_\ell\le y_l\mid G_{\ell-1} \le y_{\ell-1})} \le c_5 \gamma^{\ell}\log(P^{-1}).
\end{equation*}
\end{lemma}
\begin{proof}
Applying the subset property we obtain
\begin{equation*}
\frac{1-\Pr(G_\ell\le y_l \!\mid \!G_{\ell-1} \le y_{\ell-1})}{\Pr(G_\ell\le y_\ell \!\mid\! G_{\ell-1} \le y_{\ell-1})} = \frac{\Pr(G_{\ell-1}\le y_{\ell-1})}{\Pr(G_\ell\le y_\ell)} - 1 = \frac{\Pr(G_{\ell-1}\le y_{\ell-1})-\Pr(G_\ell\le y_\ell)}{\Pr(G_\ell\le y_\ell)}.
\end{equation*}
First, note that by Assumption~\ref{ass:error_estimate1} we have 
$G_\ell \le G+\gamma^\ell\ \text{and}\ G-\gamma^{\ell-1} \le G_{\ell-1}$,
and thus 
\begin{align*}
\Pr(G_\ell\le y_\ell) &\ge \Pr(G+\gamma^\ell\le y_\ell)= \Pr(G\le y_\ell-\gamma^\ell),\\
\Pr(G_{\ell-1}\le y_{\ell-1}) &\le \Pr (G-\gamma^{\ell-1} \le y_{\ell-1}) = \Pr(G\le y_{\ell-1}+\gamma^{\ell-1}).
\end{align*}
By definition we have
$y_{\ell-1} + \gamma^\ell =  y_\ell + (\gamma^\ell+\gamma^{\ell-1})+\gamma^\ell$,
which together with \eqref{eq:subset_ratio} implies
\begin{align*}
0\le \frac{\Pr(G_{\ell-1}\le y_{\ell-1})-\Pr(G_\ell\le y_\ell)}{\Pr(G_\ell\le y_\ell)} &\le \frac{\Pr(G\le y_{\ell-1}+\gamma^{\ell-1})-\Pr(G\le y_\ell-\gamma^\ell)}{\Pr(G\le y_\ell-\gamma^\ell)}\\ &\le \tilde c (2\gamma^{\ell}+2\gamma^{\ell-1})\log(P^{-1}),
\end{align*}
where we have used that $\Pr(G_{\ell-1}\le y_{\ell-1})\ge \Pr(G_\ell\le y_\ell)$ due to Lemma~\ref{lem:subset}.
\end{proof}
We note that the dependence on $\log(P^{-1})$ in equation \eqref{eq:subset_ratio} is the weakest assumption we can take in order to obtain an improvement through our proposed multilevel subset simulation strategy. However, the bound \eqref{eq:subset_ratio} crucially depends on the underlying \correction{limit state} function $G$ and one might drop the dependence on $\log(P^{-1})$ for certain models.

As in the classical subset
simulation, we estimate the failure sets such that 
\begin{equation}\label{eq:var_bound_mlsus}
  \delta(\widehat P_\ell)^2 =\frac{1-\Pr\left(F^{\text{ML}}_\ell\mid F^{\text{ML}}_{\ell-1}\right)}{\Pr\left(F^{\text{ML}}_\ell\mid F^{\text{ML}}_{\ell-1}\right)}(1+\phi_\ell)N_\ell^{-1} \leq L^{-s}\text{TOL}^2,
\end{equation}
%
which guarantees that the relative variance of the multilevel estimator is bounded by $\text{TOL}^2$.

{We finish by stating and proving the main theoretical result of the paper on the complexity of the proposed 
adaptive multilevel subset simulation method, with and without selective refinement, under similar assumptions 
made for the single-level complexity result in Theorem~\ref{thm:complexity_sus}.} {Recall that for simplicity we have set $K=L$ and $\ell = \ell_j=j$ for $j = 1,\dots,K$.}
\begin{theorem}\label{thm:complexity_ML}
Suppose that
\begin{enumerate}
\item Assumption~\ref{ass:error_estimate1} is satisfied for some $q \ge 0$ and $\gamma \in (0,1)$ and that \eqref{eq:subset_ratio} holds,
\item the level thresholds $y_\ell$ are defined as in Lemma~\ref{lem:subset}, and
\item there exists a constant $\phi < \infty$ such that
$\phi_{\ell}\le \phi$ in \eqref{eq:var_bound_mlsus}, for all $\ell=1,\ldots,L$.
\end{enumerate}
Then, for any $\text{TOL}>0$, the maximum discretisation level $L$ and the numbers $\{N_\ell\}$ of samples on each level can be chosen such that 
\[\mathbb E\bigg[\frac{|\widehat P^{\mathrm{ML}}- P|^2}{P^2}\bigg]\le \text{TOL}^2\]
with a cost that is bounded by
\begin{equation*}
\mathbb E\left[\costb{\widehat P^{\mathrm{ML}}}\right]\le\begin{cases} c_7 \mathrm{TOL}^{-\max(2,(1+q))}(1+\phi)\log(P^{-1}), & s = 1,\\
c_7 \mathrm{TOL}^{-\max(2,(1+q))}\log({\mathrm{TOL}}^{-1})(1+\phi)\log(P^{-1}), & s = 2, \end{cases}
\end{equation*}
for some constant $c_7>0$. Here, $s=1$ if the estimators $\widehat P_{\ell}$, $\ell=1,\ldots,L$, are uncorrelated and $s=2$ if they are fully correlated.

If in addition Assumption~\ref{ass:selective_refinement} is satisfied, the cost can be bounded by 
\begin{equation*}
\mathbb E\left[\costb{\widehat P^{\mathrm{ML}}}\right]\le \begin{cases} c_7 \mathrm{TOL}^{-\max(2,q)}(1+\phi)\log(P^{-1}), & s=1, \\
c_7 \mathrm{TOL}^{-\max(2,q)}\log({\mathrm{TOL}}^{-1})(1+\phi)\log(P^{-1}), & s=2. \end{cases}
\end{equation*}
\end{theorem}

\begin{proof}
The result follows similarly as the proof of Theorem~\ref{thm:complexity_sus} for single-level subset simulation. Note that the resulting number of samples \[
N_\ell\propto  \frac{P_L^2}{P^2} \frac{1-\Pr\left(F^{\text{ML}}_\ell\mid F^{\text{ML}}_{\ell-1}\right)}{\Pr\left(F^{\text{ML}}_\ell\mid F^{\text{ML}}_{\ell-1}\right)}(1+\phi_\ell)L^{s}\text{TOL}^{-2}
\] 
are now level dependent due to the fact that the probabilities $\mathbb P(F^{\text{ML}}_\ell\mid F^{\text{ML}}_{\ell-1})$ differ in $\ell$.  Hence, applying Lemma~\ref{lem:level_bound} the total computational costs result in
\begin{equation*}
  \begin{aligned}
  \mathbb E\left[\costb{\widehat P^{\mathrm{ML}}}\right] &\le \mathrm{TOL}^{-2}L^{s-1}(1+\phi)\sum_{\ell=1}^L\Pr(F^{\text{ML}}_\ell|F^{\text{ML}}_{\ell-1})^{-1}(1-\Pr(F^{\text{ML}}_\ell|F^{\text{ML}}_{\ell-1}))\gamma^
{-\ell q},\\
	&\le \mathrm{TOL}^{-2}L^{s-1}(1+\phi)\sum_{l=1}^L c_5\gamma^{(1-q)\ell}\log(P^{-1})\\ &\le c_6\mathrm{TOL}^{-2}L^{s-1}(1+\phi)(1+\gamma^{(1-q)L})\log(P^{-1})\\ 
	&\le c_7 \mathrm{TOL}^{-\max(2,1+q)}L^{s-1}(1+\phi)\log(P^{-1}),
  \end{aligned}
\end{equation*}
for some $c_5, c_6, c_7>0$. 
If in addition Assumption~\ref{ass:selective_refinement} holds the bound can be improved to
\[
  \mathbb E\left[\costb{\widehat P^{\mathrm{ML}}}\right] \le c_7 \mathrm{TOL}^{-\max(2,q)}L^{s-1}(1+\phi)\log(P^{-1}).
\]
\end{proof}
Clearly, the asymptotic complexity is significantly improved over classical subset simulation. In addition to the gains due to the level-dependent cost for each sample and to the variance reduction on the rarer subsets, an additional cost reduction in practice comes from the fact that the
accept/reject step for $\ell>0$ in Algorithm~\ref{alg:POP2} is computed to tolerance $\gamma^{\ell-1}$ and only accepted samples are then computed also to tolerance $\gamma^\ell$. 
The intermediate failure thresholds and thus the failure sets $F^{\text{ML}}_\ell$ are defined a priori based on the value of $\gamma$. Thus, the probabilities $\Pr(F^{\text{ML}}_\ell|F^{\text{ML}}_{\ell-1})$ are problem dependent and -- as in MLMC \cite{Giles08} -- optimal sample sizes $\{N_\ell\}$ are difficult to compute. This would be an interesting area for future investigation.


\section{Numerical Results}
\label{s:numerics}

In the following, we consider three numerical examples with increasing difficulty. We start the experiments with a one-dimensional toy example where we can ensure Assumption~\ref{ass:selective_refinement}. In this example, it is possible to verify the complexity results expected from Theorem~\ref{thm:complexity_sus} and Theorem~\ref{thm:complexity_ML} respectively. In our second example we consider a rare event estimation problem based on a Brownian motion. In this case, Assumption~\ref{ass:selective_refinement} does not hold almost surely but only in $L^p$.  The results of our numerical experiments remain promising and we observe a significant improvement through our multilevel subset simulation and the incorporation of selective refinement. The last experiment is based on an elliptic PDE model and represents a more realistic scenario of application. 

\subsection{Example 1: Toy experiment}

We start by verifying our derived complexity results on a simplified toy model. We assume that $G\sim\mathcal N(0,1)$ and define the pointwise approximation $G_\ell(\omega):= G(\omega)+\kappa(\omega)\gamma^\ell$ with $\kappa\sim\mathcal U(\{-1,1\})$ such that \eqref{eq:error_bound} is obviously satisfied. \correction{Furthermore, we let $\gamma=1/2$ and assume that $q = 2$ for the expected costs in \eqref{eq:cost1}.}
 
 The aim is to estimate the failure probability 
\begin{equation}\label{eq:ref_prob}
 \Pr(G\le -3.8) \approx 7.23\cdot 10^{-5}.
 \end{equation}

\update{Applying Algorithm~\ref{alg:selective} \cite[Alg.~1]{EHM14} allows to simulate $G_\ell^y$ satisfying Assumption~\ref{ass:selective_refinement}. For any accuracy level $\ell$, the selective refinement algorithm starts with the coarsest approximation $k=1$ and successively refines the accuracy by increasing $k$ until $|G_{k}(\omega)-y|\ge \gamma^k$ or $|G_{k}(\omega)-G(\omega)|\le \gamma^\ell$ is satisfied. Note that in more realistic applications, 
estimates of the error $|G_{k}(\omega)-G(\omega)|$ are needed, cf.~Section~\ref{ssec:num_pde} for more details. In this simplified example we increase the accuracy until $|G_{k}(\omega)-y|\ge \gamma^k$ or until $k = \ell$, since we know that  $|G_\ell(\omega)-G(\omega)| \le \gamma^\ell$ by definition. We then set $G_\ell^y(\omega)=G_{k}(\omega)$.}

We compare standard MC, classical subset simulation with and without selective refinement, as well as our proposed adaptive multilevel subset simulation algorithm, choosing $L\propto \log({\mathrm{TOL}}^{-1})$ in all cases, where $\mathrm{TOL}$ is the required tolerance.
For subset simulation, we use $K=5$ subsets such that all $\Pr(F_j\mid F_{j-1})\approx[0.1,0.2]$, where we choose 
the threshold values $(y_5,y_4,y_3,y_2,y_1,y_0) = (-3.8,-3.3,-2.8,-2,-1.3,\;\infty)$. In contrast, for multilevel  subset simulation we choose $y_\ell$ following Lemma~\ref{lem:subset} and let the number of subsets increase depending on the size of $L$. For the estimation of \correction{$\widehat P_1$} in \eqref{eq:ML_estimator} we apply a standard MC estimate.  The highest accuracy level $L$ and the numbers of samples on each subset ($N_j$ and $N_\ell$ resp.) are chosen according to the assumptions in Theorems~\ref{thm:complexity_sus} and \ref{thm:complexity_ML}, respectively. 

Fig.\ \ref{fig:toyexample:cost_tol} shows the estimated \correction{rRMSE}~$\delta(\widehat P)$ using the true reference probability of equation~\eqref{eq:ref_prob} plotted against the computational cost for the different applied estimators. We have used $100$ runs for building the estimates of $\delta(\widehat P)$ for each algorithm.
\begin{figure}[t]
\centering \includegraphics[width=0.65\textwidth]{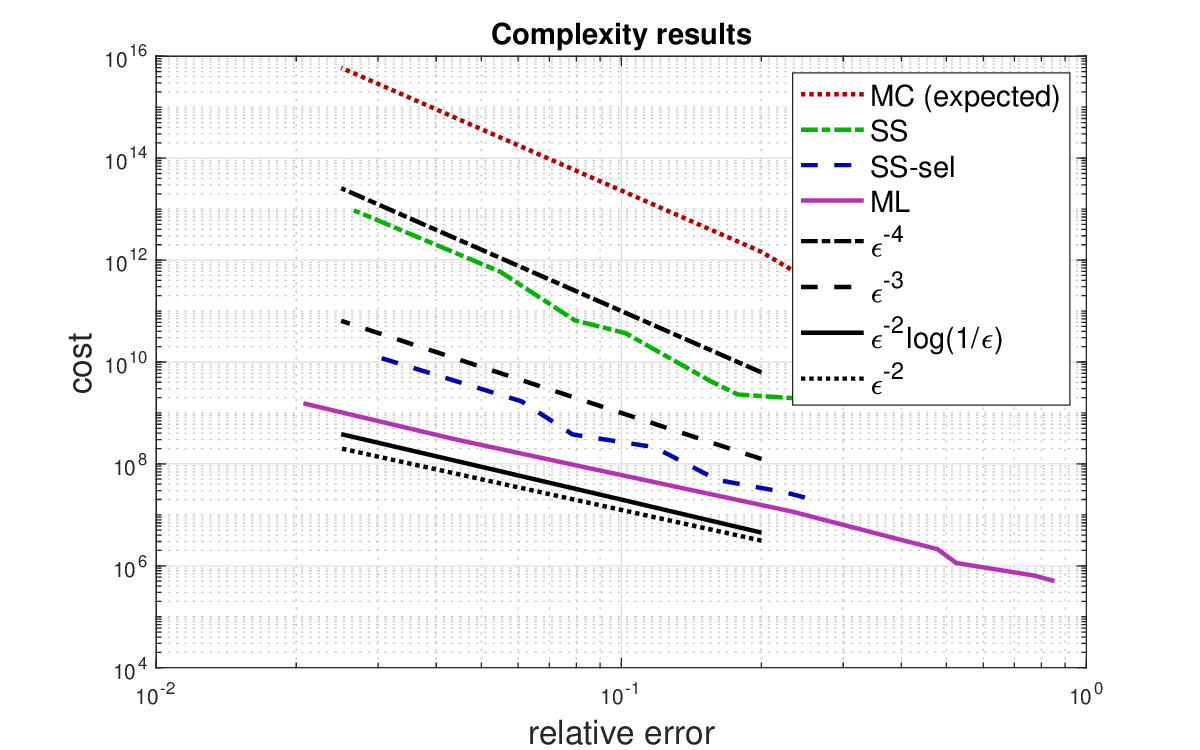}
\caption{Mean computational cost plotted against the estimated values of $\delta(\widehat P)$ (with tolerance ${\mathrm{TOL}}$) for the toy problem in Example 1 using standard MC (expected), classical subset simulation, as well as subset simulation and multilevel subset simulation with selective refinement. }\label{fig:toyexample:cost_tol}
\end{figure} 

\subsection{Example 2: Brownian motion}

In our second numerical example, we consider estimating the probability that a standard Brownian motion drops below a threshold value within a certain time interval. To be more precise, let $(B_t)_{t\ge0}$ be a Brownian motion. We are interested in the estimation of
\[\Pr(F) = \Pr\big(\min\nolimits_{t\in[0,1]}\, B_t \le -4\big),\]
where we can compute the reference value of $\Pr(F)$ by
\[\Pr\big(\min\nolimits_{t\in[0,1]}\, B_t \le -4\big) = 2\cdot \Pr(B_1\le -4) \approx 6.3\cdot 10^{-5}.\]
We define the limit state function pointwise by
\[G(\omega) := \min\nolimits_{t\in[0,1]}\, B_t(\omega) + 4,\]
and introduce approximations of the limit state function such that
\[
\update{\tilde G_k(\omega) := \min\nolimits_{t\in T_k}\, B_t(\omega)+4,\quad T_k := \Big\{\frac{i}{2^k}\mid i=0,\dots,2^k\Big\}.}
\]

\update{For all $1 \le p < \infty$, the resulting approximation 
error can be bounded in $L^p$ \cite{RITTER1990337}, i.e., there exists a constant $C_p>0$ such that
\begin{equation}\label{eq:errorbound_BM}
\|G - \tilde G_k\|_{L^p(\Omega)} = 
\mathbb E\Big[\big|\min\nolimits_{t\in[0,1]}\, B_t(\omega)- \min\nolimits_{t\in T_k}\, B_t(\omega)\big|^p\Big]^{1/p} \le C_p 2^{-k/2}\,.
\end{equation}}
The Brownian motion is generated path-wise through a Karhunen-Lo\'eve expansion 
\begin{equation*}
B_t(\omega) = \sum_{i=1}^\infty \xi_i(\omega) \varphi_i(t), 
\end{equation*}
with  $\xi_i \overset{\text{i.i.d.}}{\sim} \mathcal N(0,\frac{1}{(i-1/2)^2})$ and $\varphi_i(t) = \frac{\sqrt{2}}{\pi}\sin((i-1/2)\pi t)$. See  Fig.~\ref{fig:BM_paths} for various realizations of the Brownian motion,
conditioned on the different chosen subsets, i.e.~for different level thresholds $y_j$.
\begin{figure}[t]
\centering \includegraphics[width=0.75\textwidth]{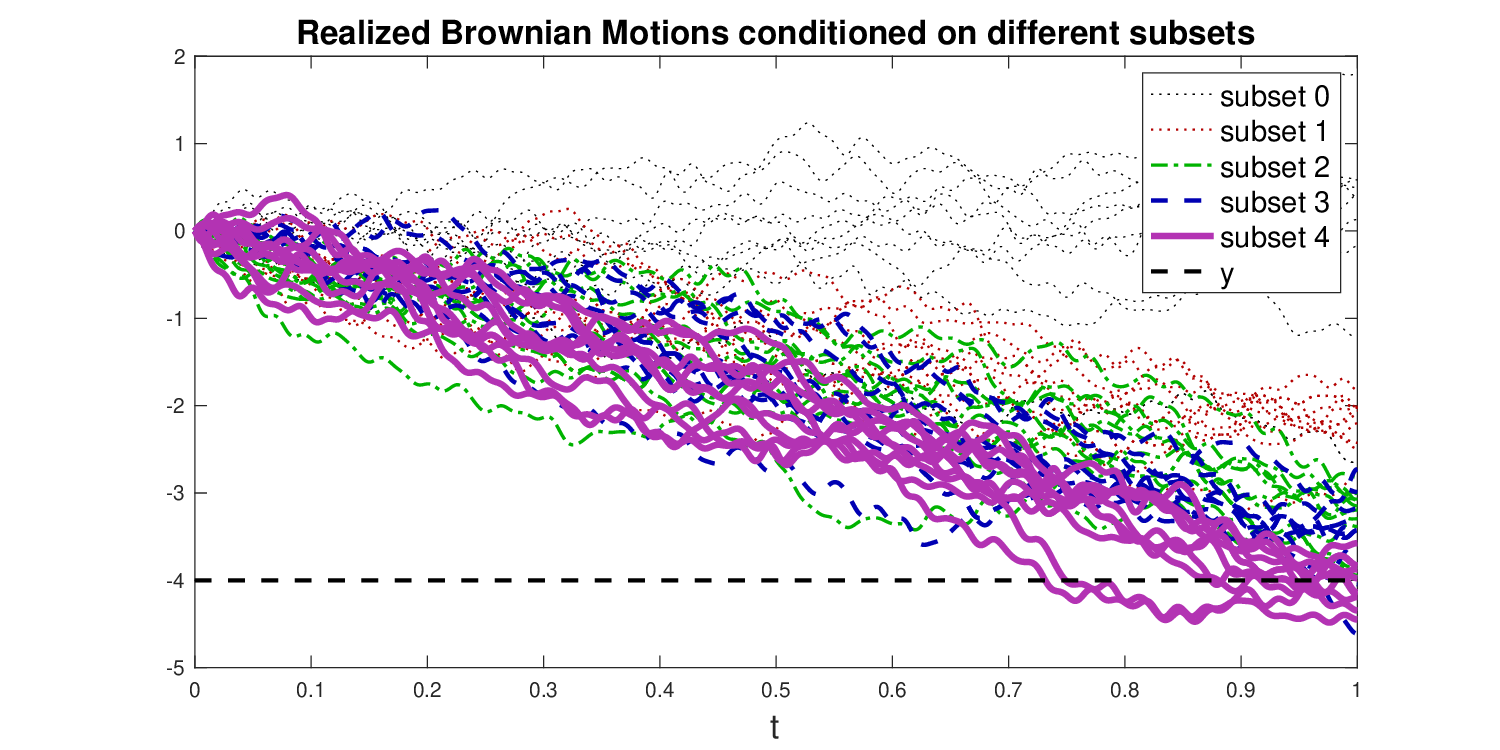}
\caption{Paths of the Brownian motion conditioned on the different subsets $F_j$. }\label{fig:BM_paths}
\end{figure}

\update{As in Example 1, on each accuracy level $\ell$ and for each sample $\omega$, the selective refinement algorithm starts on the coarsest level $k=1$ and successively refines the accuracy by increasing $k$ until $|\tilde G_{k}(\omega)-y|\ge \gamma^k$ or $|\tilde G_{k}(\omega)-G(\omega)|\le \gamma^\ell$. We then set $G_{\ell}^{y}(\omega) := \tilde G_{k}(\omega)$. Based on the error bound \eqref{eq:errorbound_BM}, we consider $\gamma = \frac{1}{\sqrt{2}}$ in our numerical experiments. Then $q=2$ and the latter of the two conditions above is satisfied at least in an $L^p$-sense for $k = \ell + 2\log_2(C_p)$. Unfortunately it is not possible to satisfy this condition in an almost sure sense.

Depending on the required tolerance value ${\mathrm{TOL}}$ we set again $L \propto \log({\mathrm{TOL}}^{-1})$ and fix the number of subsets in the classical subset simulation to $K_{\mathrm{SL}}=5$. In the multilevel formulation, $K_{\mathrm{ML}} =\max(L-1,8)$ subsets are considered and the threshold values $y_j$ are chosen again according to Lemma~\ref{lem:subset}, 
with the following values of $\ell_j$ for the failure sets $F_j = \{\omega: G_{\ell_j}^{y_j}(\omega) \le y_j\}$:}
\begin{table}[H]
    \centering
    \begin{tabular}{c|ccccccc}
      $j$  & 1   & 2     & 3  & 4 & 5 & \dots & $L-1$ \\ \hline
      $\ell_j$ & 4     & 4    & 4   & 5 & 6 & \dots & L\\
    \end{tabular}
    \label{tab:levels_BM}
  \end{table}

\noindent 
The MC estimates for the multilevel estimator $\text{ML}$ are built using 100 paths, resulting in
\[\correction{\mathbb E}[\widehat P^{\text{ML}}] \approx 5.11\cdot 10^{-5}\quad\text{and}\quad \delta(\widehat P^{\text{ML}}) \approx 0.0534 \leq {\mathrm{TOL}} =: 0.1\, .\]

To compare the proposed multilevel method with classical subset simulation in Fig.~\ref{fig:cost_tol} we compare the resulting costs for various choices of 
${\mathrm{TOL}}$.  Note that we have estimated the expected number of samples for the classical MC estimator as $N=({\mathrm{TOL}})^{-2} P^{-1}$.

\begin{figure}[t]
\centering \includegraphics[width=0.65\textwidth]{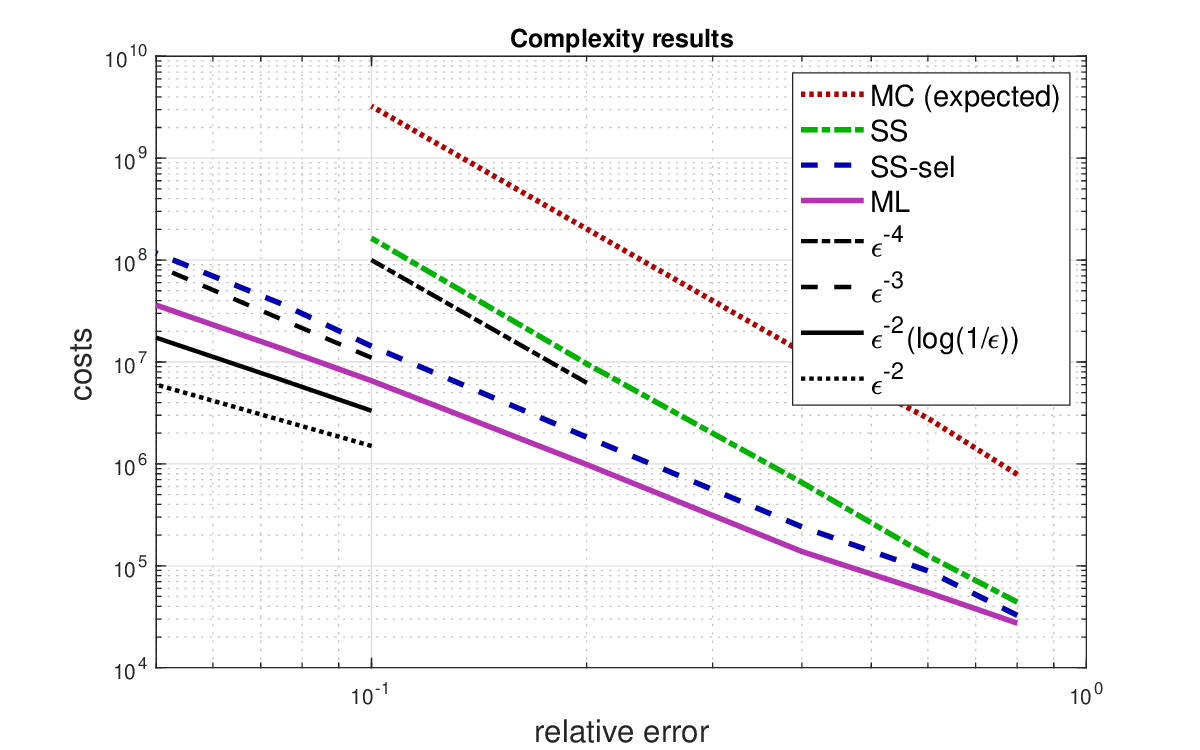}
\caption{Mean computational cost plotted against the estimated values of $\delta(\widehat P)$ (with tolerance ${\mathrm{TOL}}$) for the Brownian motion in Example 2
using standard MC (expected), classical subset simulation, as well as subset simulation and multilevel subset simulation with selective refinement.}\label{fig:cost_tol}
\end{figure}

\subsection{Example 3: Elliptic PDE}
\label{ssec:num_pde}

Finally, we consider the following diffusion equation, which is used, e.g., to model stationary Darcy flow:
\begin{equation}\label{eq:strong1}
  \begin{aligned}
    -\nabla\cdot A\nabla u &= 0 \quad\text{in }\mathcal{D},\\
    \text{subject to} \quad u  = 0 \quad\text{on }\Gamma_{1}, \quad
    u & = 1 \quad\text{on }\Gamma_{2}, \quad
    \nu\cdot A\nabla u = 0 \quad\text{on }\Gamma_{3}\cup\Gamma_4,\\
  \end{aligned}
\end{equation}
where $\mathcal D=(0,1)\times(0,1)$ is a unit square and 
$\Gamma_1$, $\Gamma_2$,
$\Gamma_3$, $\Gamma_4$ are the left, right, upper, and lower boundaries,
respectively. The
permeability $A(x,\omega)$ is a log-normal random field. In particular, $\log(A(x,\omega))$ is
a stationary, zero mean Gaussian field $\mathcal N(0,\mathcal C)$ with covariance operator $\mathcal C = (-\Delta+\tau^2\cdot {\mathrm{id}})^{-\alpha}$, where $\Delta$ denotes the Laplacian operator equipped with Neumann boundary conditions and we set $\tau =0.1$ and $\alpha = 1$. The random field has been generated path-wise via a 
truncated Karhunen-Lo\'eve expansion, 
see Fig.~\ref{fig:realization} for two realizations of the random field and the associated PDE solutions. 

\update{The functional to define the limit state function is chosen to be
\begin{equation}\label{eq:lsf_cont}
  \mathcal G(u)= y - \frac1{|B|}\int_{B}u(x)\mathrm{d}x,
\end{equation}}
for $B = [0.4,0.6]\times [0.9,0.99]\subset \mathcal D$, i.e.,  `failure' occurs when the 
mean of $u$ over the subdomain $B$ exceeds $y$, with $|\cdot|$ denoting the area of $B$. \correction{For the numerical experiments below we choose $y=0.92$.}
\begin{figure}[t]\label{fig:realization}
\centering
\includegraphics[width=1\textwidth]{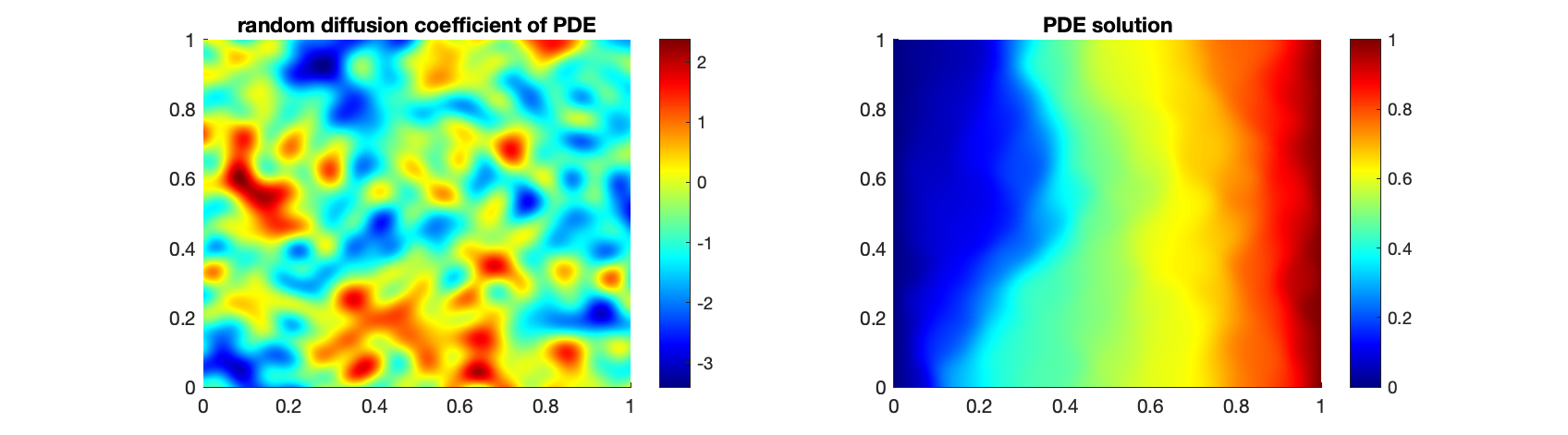}\\\includegraphics[width=1\textwidth]{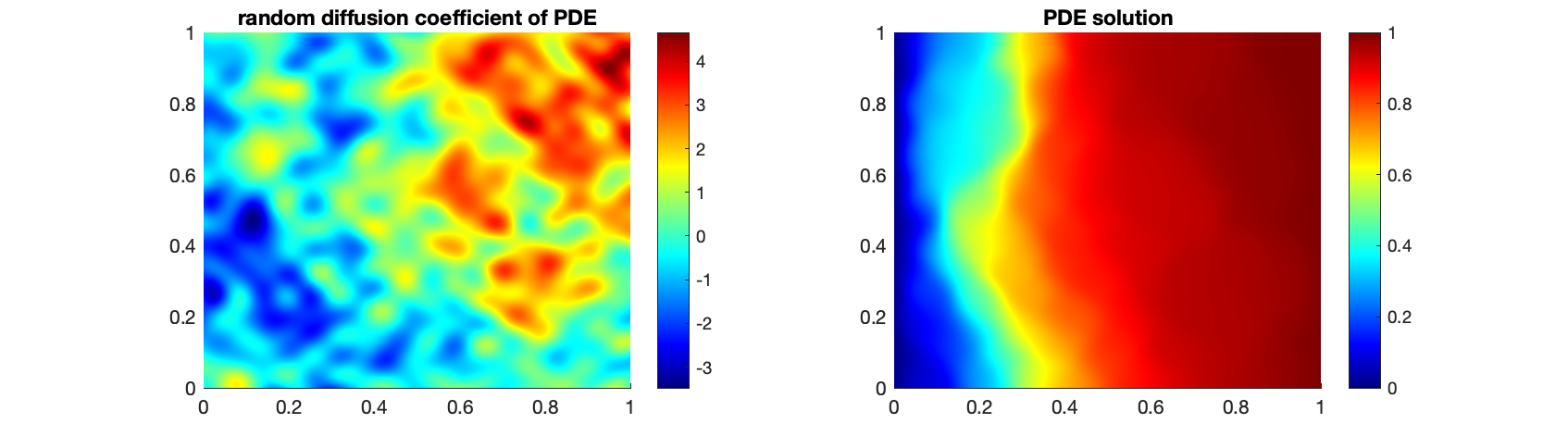}\vspace{-0.5cm}
   \caption{Two realizations of the log-normal permeability field $A(x,\omega)$ (left) and the corresponding solution for the Darcy flow problem \eqref{eq:strong1} (right) in Example 3.  The first realization is sampled from the whole probability space, whereas the second one is a rare event realization.}
\end{figure}

Now, given a sample $\omega \in \Omega$ and defining
\[
H^1_D(\mathcal{D})=\{v\in H^1(\mathcal{D}): v|_{\Gamma_1}=0, 
\ v|_{\Gamma_2}=1\} \ \ \text{and} \ \
H^1_0(\mathcal{D})=\{v\in H^1(\mathcal{D}): v|_{\Gamma_1 \cup \Gamma_2}=0\}, 
\]
the weak form of \eqref{eq:strong1} is equivalent to finding
$u \in H^1_D(\mathcal{D})$ such that
\begin{equation}\label{eq:weak}
  a(\omega;u,v) = \int_{\mathcal{D}}A(x,\omega)\nabla u \cdot\nabla v\dx = 0,
  \quad\text{for all }v\in H^1_0(\mathcal{D}).
\end{equation}
\update{The value of the exact limit state function at $\omega$ is set to be $G(\omega) = \mathcal G(u)$.}

We approximate the weak formulation and the limit state function
with a finite element (FE) method. Let $\mathcal{T}_h$ be a uniform
triangulation of $\mathcal{D}$, and suppose $\mathcal{V}_h$ denotes the associated $P_1$-Lagrange
FE space. The FE approximation of the solution $u$ of \eqref{eq:weak} on $\mathcal{T}_h$ is then defined to be the
 $u_h\in \mathcal{V}_h\cap H^1_D(\mathcal{D})$ satisfying
\begin{equation}
  a(\omega;u_h,v_h) = 0\quad\text{for all }v_h\in \mathcal{V}_h\cap H^1_0(\mathcal{D}).
\end{equation}
\update{The FE approximation of the limit state function is defined as $\tilde G_h(\omega) = \mathcal G(u_h)$.

Given $\gamma \in (0,1)$, we then compute $G_\ell(\omega)$ by iterating over a sequence of uniformly
refined meshes $\mathcal{T}_h$ with $h \to 0$, starting from some $h_0>0$, until
   $|G(\omega)-\tilde G_h(\omega)|\leq\gamma^\ell$.
For selective refinement we start with accuracy $k=1$ and reduce $h$ until $|G(\omega)-\tilde G_h(\omega)|\leq\gamma^k$ or $|\tilde G_h(\omega)-y|\geq \gamma^k$. If the second condition is satisfied we stop and set $G^y_\ell(\omega) = \tilde G_h(\omega)$. Otherwise we increment $k \to k+1$ and repeat the process until the second condition is satisfied for some $k$ or until $k = \ell$.
}

To estimate the error in the limit state function \correction{$\tilde G_h$ computed on mesh $\mathcal{T}_h$ we use a hierarchical a-posteriori error estimator. In particular, for the chosen random field and limit state function we may assume that
\begin{equation}
\label{eq:FE_error_bounds}
ch^2 \le |G(\omega)-\tilde G_h(\omega)| \le Ch^2, \quad \text{for some} \ \ 0 < c \le C < \infty.
\end{equation}
Provided $r := C/c < 4$ it follows from those bounds and the triangle inequality 
that 
\[
|G(\omega)-\tilde G_h(\omega)|
\le 
\,\frac1{1- r/4}\,
|\tilde G_{h/2}(\omega)-\tilde G_h(\omega)| \update{=: \eta_h\,. }
\]}
The \texttt{PDE Toolbox} within \texttt{MATLAB} is used in our implementation. The FE meshes that were constructed are shown in
Fig.~\ref{fig:mesh}.
For the numerical experiments we choose $\gamma =1/4$ \update{and use the sufficient condition $\eta_h \le \gamma^\ell$ to bound the FE error on $\mathcal{T}_h$. In the experiments, the estimate $\eta_h$ of the FE error on the finest mesh depicted in Fig.~\ref{fig:mesh} was always below $\gamma^4$.}
\begin{figure}[t]\label{fig:mesh}
\centering
\includegraphics[width=\textwidth]{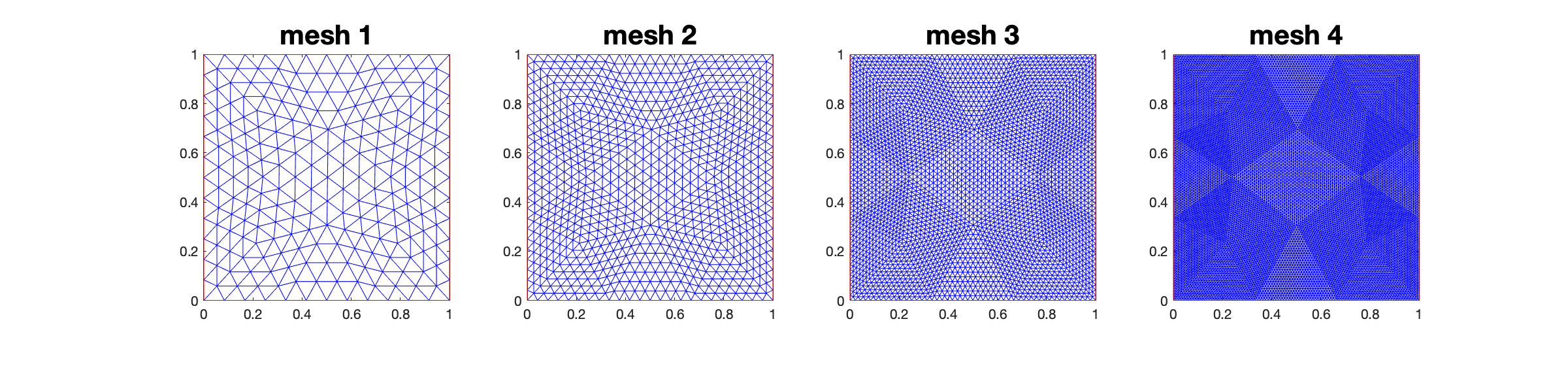}\vspace{-1cm}
   \caption{Different levels of refinement of the FE mesh 
   for 
   the Darcy flow problem \eqref{eq:strong1} in Example 3.}
\end{figure}

For the proposed single-level and multilevel estimators, we choose $L$, $N_j$, $N_\ell$ according to Theorems~\ref{thm:complexity_sus} and \ref{thm:complexity_ML}, with a fixed number of subsets $K=4$ for the classical (single level) subset simulation. The number of subsets for our multilevel subset simulation on the highest evaluated accuracy, together with the corresponding accuracy levels $\ell_j$, are presented in Table~\ref{tab:ml1} (left). Moreover, the Table shows the mean probability for each level and the mean acceptance rate for the Markov chains. The details for single-level subset simulation are presented in Table~\ref{tab:ml1} (right). 

\begin{table}[t]
  \centering
  \begin{tabular}{c|ccccc}
    $j$ & $1$& $2$& $3$ & $4$ & $5$  \\ \hline
    $\ell_j$ & $2$ & $2$ & $2$ & $3$ & $4$ \\
    $P_j$ & $0.05$ & $0.13$ & $0.17$ & $0.38$ & $0.83$  \\
    $\alpha_j$ & --- &$0.45$ & $0.32$ & $0.22$ & $0.18$
  \end{tabular} \qquad\qquad
  \begin{tabular}{c|ccccc}
    $j$ & $1$& $2$& $3$ & $4$  \\ \hline
    $\ell_j$ & $4$ & $4$ & $4$ & $4$  \\
    $P_j$ & $0.13$ & $0.1$ & $0.16$ & $0.07$   \\
    $\alpha_j$ & --- &$0.5$ & $0.38$ & $0.29$ 
  \end{tabular}\medskip
  \caption{Chosen level $(\ell_j)$, mean probability $(P_j)$ and mean acceptance rates $(\alpha_j)$ for each subset $j$ after $100$ runs of 
  adaptive multilevel subset 
simulation 
(left) and 
single-level subset simulation 
(right), both with selective refinement.}
  \label{tab:ml1}
\end{table}
After $100$ runs of the multilevel estimator $\text{ML}$ in this setting \update{with $L=4$}, the estimated mean rare event probability and the relative variance are
\begin{equation}
  \correction{\mathbb E}[\widehat P^{\mathrm{ML}}] \approx 1.8\cdot 10^{-4}\quad\text{and}\quad 
  \delta\big(\widehat P^{\mathrm{ML}}\big) \approx 0.172,
\end{equation}
while for the classical subset simulation estimator $\text{SuS}$ they are
\begin{equation}
  \correction{\mathbb E}[\widehat P^{\mathrm{SuS}}] \approx 1.5\cdot 10^{-4}\quad\text{and}\quad \update{\delta\big(\widehat P^{\mathrm{SuS}}\big) \approx 0.192}.
\end{equation}

For each level $\ell$ in the multilevel subset simulation, the mean number of samples computed on each mesh level $k$ are 
presented in Table~\ref{tab:number1} (left). In the selective refinement process, the computation for each sample always starts on mesh level $k=1$ and the numbers on the lower mesh levels are the cumulative ones.  \update{The vast majority of samples are generated from inexpensive low-resolution simulations, or put in another way, our multilevel estimator requires the equivalent cost of $\sim 175$ PDE solves on the finest accuracy level to estimate a rare event of $\mathcal{O}(10^{-4})$ with a relative variance of $\delta\big(\widehat P^{\mathrm{ML}}\big) \approx 0.17$.
Since the chosen regime is not yet optimally balancing the cost, there is still potential for significant further gains.}
\begin{table}[t]
  \centering
  \begin{tabular}{c|ccccc}
    $j \backslash k$ & 1 & 2 & 3 & 4  \\ \hline
    $1$  & $3602$ & $610.7$  & $0$     & $0$       \\
    $2$  & $3602$  & $1902.9$ & $0$  & $0$         \\ 
    $3$  & $3602$  & $1564$ & $0$ & $0$      \\ 
    $4$  & $2002$  & $1071.6$ & $326.2$ & $0$  \\ 
    $5$ & $502$ & $281$ & $108.5$ & $15.2$ \\\hline
    \text{Total} & $13310$ & $5430.2$ & $434.7$ & $15.2$\\\hline
    $c_k$ & $1$ & $\update{8}$ & $\update{64}$  & $\update{512}$
  \end{tabular}\quad 
   \update{\begin{tabular}{c|ccccc}
    $j \backslash k$ & 1 & 2 & 3 & 4  \\ \hline
    $1$  & $3602$ & $1179.8$  & $381.4$     & $66.4$       \\
    $2$  & $3602$  & $2518.6$ & $1272.1$  & $247.6$         \\ 
    $3$  & $3602$  & $2373.3$ & $1271.1$ & $252.1$      \\ 
    $4$  & $3602$  & $2220.3$ & $1066.8$ & $181.8$ \\ \ \\ \hline
    \text{Total} & $14404$ & $8291.9$ & $3991.4$ & $747.9$\\\hline
    $c_k$ & $1$ & $\update{8}$ & $\update{64}$  & $\update{512}$
  \end{tabular}\medskip
  }
  \caption{Mean number of samples computed on mesh level $k$ for subset $j$, for adaptive multilevel subset simulation 
  (left) 
  and single-level subset simulation (right), both with selective refinement, estimated from $100$ runs \update{with $L=4$} and summed up over all mesh levels. The tables also show the total number of samples and the (normalized) theoretical cost per sample $c_k$ on each mesh level \update{assuming $q=3/2$}.}
  \label{tab:number1}
\end{table}

\update{The table also shows the normalized theoretical cost $c_k$ 
for one sample on mesh level~$k$ (including the cost for the error estimation). Here, we have assumed a sparse direct solver such as \texttt{CHOLMOD} \cite{Chen08} (i.e. backslash in Matlab) to solve each of the arising FE systems. Theoretically, the cost for a sparse direct solver like \texttt{CHOLMOD} applied to a 2D FE system grows like $\mathcal{O}(n_{\text{FE}}^{3/2})$, while the number of unknowns grows like $n_{\text{FE}} = \mathcal{O}(4^{k})$ under uniform refinement. Hence, we have chosen $q = 3/2$ in the definition of the normalized theoretical cost $c_k$. 
%
Based on this assumption, a comparison of the total normalized cost for all the estimators is presented in Fig.\ \ref{fig:pde_cost_tol} (left) for various choices of the tolerance $\text{TOL}$. 
The vastly superior performance of the proposed multilevel estimator is again apparent, e.g., for an estimate with $\delta(\widehat P)\approx 0.25$ there is a more than 10-fold gain in efficiency compared to subset simulation with selective refinement and a more than 60-fold gain compared to standard subset simulation with all samples being computed to accuracy $\gamma^L$. For a fair comparison, standard subset simulation is also ran with the hierarchical error estimator to guarantee the required accuracy for each sample. Asymptotically, the cost does appear to grow as predicted in Theorem~\ref{thm:complexity_ML}, even though  (as in Example 2) Assumption~\ref{ass:selective_refinement} does not hold uniformly in $\omega$ here. Note that as predicted in Theorem~\ref{thm:complexity_sus}, the selective refinement approach alone also leads to clear gains over standard subset simulation and to a better asymptotic rate. 
\begin{figure}[t]
\centering
\hspace*{-0.3cm}\includegraphics[width=0.52\textwidth]{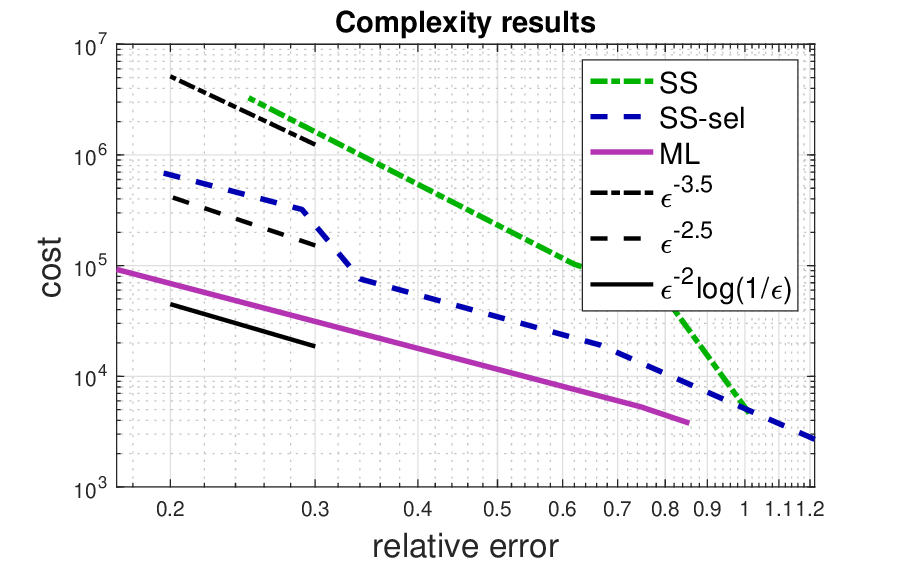}\includegraphics[width=0.52\textwidth]{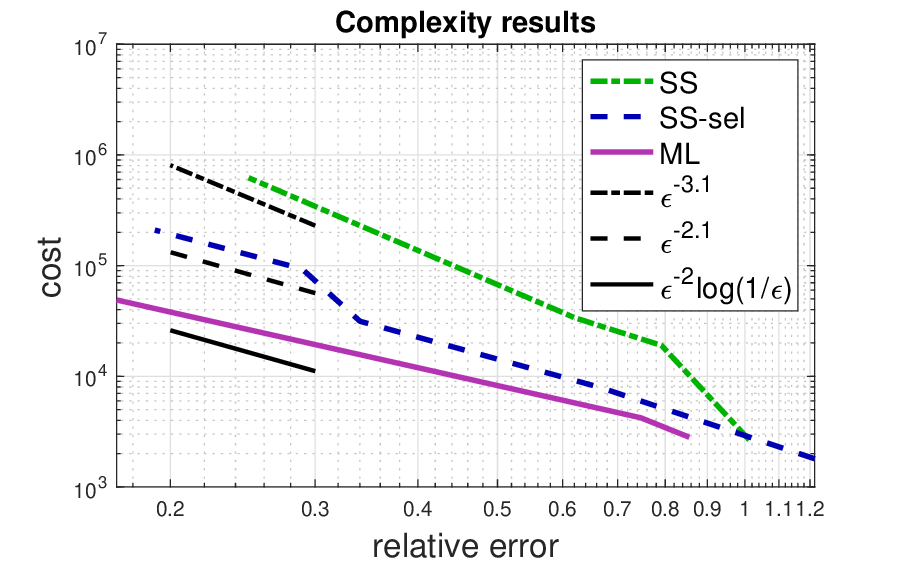}\hspace*{-0.3cm}
\caption{Mean computational cost plotted against the estimated values of $\delta(\widehat P)$ for Example 3, the Darcy problem in \eqref{eq:strong1}, \update{obtained by varying ${\mathrm{TOL}}$ for $q=3/2$ (left) and $q=1.1$ (right)} using classical subset simulation, as well as subset simulation and adaptive multilevel subset simulation with selective refinement.}\label{fig:pde_cost_tol}
\end{figure} 

In practice, the cost of \texttt{CHOLMOD} applied to a 2D FE system is often observed to grow only like $\mathcal{O}(n_{\text{FE}}^{1.1})$ (at least in the pre-asymptotic regime). Thus, in Fig.\ \ref{fig:pde_cost_tol} (right) we also plot the gains for $q=1.1$. Even in this case, we still see a 20-fold gain in efficiency of the proposed multilevel estimator compared to standard subset simulation. 
The gains for both the single- and the multilevel method with selective refinement are even more dramatic when $q>3/2$, e.g., for 3D problems or for rougher random coefficients.}

\section{Conclusions}
\label{s:conclusions}

In this paper, we propose a new multilevel subset simulation estimator of the probability of rare events. By constructing a hierarchy of numerical approximations to a model of a complex physical process, and by using a posteriori error estimation, the subset property of the intermediate failure domains is preserved. The estimator was tested in a Darcy flow problem, for which {it reduced the cost compared to the classical subset simulation estimator by more than a \update{factor of 60} for a practically relevant relative error tolerance of 25\%}. Given the wide applicability of subset simulation, several problems beyond the simulation of rare events may also benefit from this dramatic increase in efficiency.

\section*{Acknowledgements}
The work of RS is supported by the Deutsche Forschungsgemeinschaft (DFG, German Research Foundation) under Germany’s Excellence Strategy EXC 2181/1 - 390900948 (the Heidelberg STRUCTURES Excellence Cluster). RS and SW acknowledge the support of the Erwin Schrödinger Institute. 

\bibliographystyle{siam}
\bibliography{references}

\newpage

\begin{appendix}

\section{Algorithmic details}
\label{app:algorithm}
\update{We provide here some details on our versions of subset simulation, subset simulation with selective refinement and the adaptive multilevel subset simulation algorithm based on
shaking transformations and
 the POP Algorithm~\ref{alg:POP2}, as presented in Section~\ref{sec:MCMC}. The limit state function~$G$, as well as the corresponding numerical approximations $G_\ell$ are modelled 
 via measurable Gaussian transformations $T$ and $T_\ell$ such that
\[G(\omega) = T(\Theta(\omega))\quad \text{and} \quad G_\ell(\omega) = T_\ell(\Theta(\omega)),\]
where $T,T_\ell: \mathbb R^d\to\mathbb R$ are measurable mappings and $\Theta\sim\mathcal N(0,I_{\mathbb R^d})$. 
They only differ in the construction of the subsets $F_j= \{\Theta\in A_j\}\subset \Sigma$ and $A_j\subset\mathcal B(\mathbb R^d)$.}
\paragraph{Subset simulation} \update{In order to apply shaking transformations within subset simulation, the subsets are represented in the form of \eqref{eq:subsets_GT} with a fixed level of accuracy $L$ and
\[F_j = \{\omega\in\Omega:\ G_L(\omega)\le y_j\},\quad \text{for}\quad j=1,\dots,K. \]
We define $\varphi(\theta,y) = T_L(\theta)-y$ and apply the specific variant of POP based on a Gaussian transformation presented in Algorithm~\ref{alg:classicalsus}.}
\begin{algorithm}[t]
\caption{\update{Subset simulation based on POP -- with and without selective refinement.}}
\label{alg:classicalsus}
\update{\begin{algorithmic}[1]
\State Given an accuracy level $L$, a correlation parameter $\eta\in[0,1]$ and a seed $\Theta_{0,0}\sim \Theta$, define 
subsets $F_j = \{\Theta\in A_j\}$ for all $j=1,\dots,K$ such that
\[
A_j = 
\begin{cases}
\{\theta\in\mathbb R^d\mid T_L(\theta)\le y_j\},& \text{for classical subset simulation}, \\[0.5ex]
\{\theta\in\mathbb R^d\mid T_L^{y_j}(\theta)\le y_j\}, &  \text{for selective refinement if}\ y_{j-1}-y_j\ge 2\gamma^L, 
\end{cases}
\]
as well as a shaking transformation $S_\eta(\theta,w)$ 
and
a rejection operator $M_j^{S_\eta}$ according to \eqref{eq:pcn_shaking} and \eqref{eq:rejection_operator}, respectively.\vspace{1ex} 
\For{$j=0,\dots,K-1$} \vspace{1ex}
\For{ $i=0,\dots,N_{j+1}-1$}
\State Generate $W_{j,i}\sim \Theta$.
\State Shake and accept/reject $\Theta_{j,i+1} = M_j^{S_\eta}(\Theta_{j,i},W_{j,i})$.
\EndFor
\State Estimate the probability of subsets $F_{j+1}$ by \ 
$\widehat P_{j+1} = \frac1{N_{j+1}}\sum\limits_{k=0}^{N_{j+1}-1} \mathbbm{1}_{A_{j+1}}(\Theta_{j,k})$.\vspace{-1ex}
\State Set $i_j = \arg\min\{k\mid \Theta_{j,k} \in A_{j+1}\}$.\vspace{1ex}
\State Define initial state for next level $\Theta_{j+1,0} = \Theta_{j,i_j}$\, .\vspace{1ex}
\EndFor\vspace{1ex}
\State \textbf{Result:} $\widehat P = \prod_{j=1}^{K} \widehat P_j$\, .
\end{algorithmic}}
\end{algorithm}

\paragraph{Subset simulation with selective refinement} \update{In order to incorporate the selective refinement strategy into subset simulation, 
we need to define the intermediate failure sets 
\[
F_j = \{\omega\in\Omega:\ G_L^{y_j}(\omega)\le y_j\},\quad \text{for}\quad j=1,\dots,K.
\]
Following Lemma~\ref{lem:subset0}, these sets satisfy the subset property provided $y_{j-1}-y_j\ge 2\gamma^L$. However, the classical strategy to choose the threshold values $y_j$ such that $\mathbb P(F_j\mid F_{j-1}) \approx p_0$ for some constant value $p_0\in[0.1,0.3]$ will not always ensure $y_{j-1}-y_j\ge 2\gamma^L$. On the other hand, it is easy to check in practice on each of the failure sets whether it 
is satisfied in order to decide whether to apply selective refinement on that set or not. 
To formulate the algorithm via POP, we introduce measurable Gaussian transformations $T_L^{y_{j}}:\mathbb R^d\to\mathbb R$ such that the $G_L^{y_j}(\omega) = T_L^{y_j}(\Theta(\omega))$ satisfy Assumption~\ref{ass:selective_refinement}. This small variant of the above subset simulation algorithm is also incorporated in Algorithm~\ref{alg:classicalsus}.}

\begin{algorithm}[t]
\caption{\update{Adaptive multilevel subset simulation.}}
\label{alg:MLsubset}
\update{\begin{algorithmic}[1]
\State Given an accuracy level $L$ and a correlation parameter $\eta\in[0,1]$, 
define\vspace{0.5ex}
\begin{itemize}
    \item subsets $F_\ell^{\mathrm{ML}} = \{\Theta\in A_\ell\}$ with 
    $A_\ell = \{\theta\in\mathbb R^d \mid T_\ell^{y_\ell}(\theta) \le y_\ell\}$, \ for $\ell =0,\dots, L$, and the sequence of failure thresholds
    \[
    y_L=0 \quad \text{and} \quad y_\ell=(\gamma^\ell+\gamma^{\ell+1}) + y_{\ell+1}, \quad \ell=1,\dots,L-1,
    \]
    \item a shaking transformation $S_\eta(\theta,w)$ as defined in \eqref{eq:pcn_shaking}, and 
    \item a rejection operator $M_j^{S_\eta}$ as defined in \eqref{eq:rejection_operator}.\vspace{0.5ex}
\end{itemize}
	\State Set $i=0$, $N_1 = 1$, $\widehat\delta(\widehat P_1) = \infty$.
 \While{$\widehat\delta(\widehat P_1)>\text{TOL}$}
	\State Generate i.i.d.~$\Theta_{0,i}\sim\Theta$.\vspace{-1ex}	
        \State Estimate the probability of subset $F^{\text{ML}}_1$ by 
 $\displaystyle \widehat P_1 = \frac1{N_1}\sum\limits_{k=0}^{N_1-1} \mathbbm{1}_{A_{1}}(\Theta_{0,k}).$ 
	\State Estimate $\delta (\widehat P_1)\approx \widehat\delta (\widehat P_1)$ and increase $i = i+1$, $N_1 = N_1+1$.
 \EndWhile
 \State Set $i_0 = \arg\min\{k\mid \Theta_{0,k} \in A_{1}\}$
 and define initial state $\Theta_{1,0} = \Theta_{0,i_0}$ for next level.
\For{$\ell=1,\dots,L-1$} 
	\State Set $i=0$, $N_{\ell+1} = 1$, $\widehat P_{\ell+1}=\mathbbm{1}_{A_{\ell+1}}(\Theta_{\ell,i})$. 
	\While{$\widehat\delta(\widehat P_{\ell+1})>\text{TOL}$}
		\State Generate $W_{\ell,i}\sim \Theta$.
		\State Shake and accept/reject $\Theta_{\ell,i+1} = M_\ell^{S_\eta}(\Theta_{\ell,i},W_{\ell,i})$.\vspace{0.5ex}
  \State Increase $i = i+1$, $N_{\ell+1} = N_{\ell+1}+1$.	
  \State Estimate the subset probability $F^{\text{ML}}_{\ell+1}$ by 
    $\displaystyle \widehat P_{\ell+1} = \frac1{N_{\ell+1}}\sum\limits_{i=0}^{N_{\ell+1}-1} \mathbbm{1}_{A_{\ell+1}}(\Theta_{\ell,i}).$ 
		\State Estimate $\delta (\widehat P_{\ell + 1})\approx \widehat\delta (\widehat P_{\ell +1})$.
	\EndWhile
	\State Set $i_\ell = \arg\min\{k\mid \Theta_{\ell,k} \in A_{\ell+1}\}$ and
        define initial state $\Theta_{\ell+1,0} = \Theta_{\ell,i_\ell}$ for next level. 
\EndFor
\State \textbf{Result:} $\widehat P^{\mathrm{ML}_1} = \prod_{\ell=1}^{L} \widehat P_\ell$\,.
\end{algorithmic}}
\end{algorithm}\paragraph{Adaptive multilevel subset simulation} 
\update{To formulate this
algorithm in such a way that the subset property is satisfied we choose 
$y_1,\dots,y_K$ as defined in Lemma~\ref{lem:subset} and choose the subsets 
\[
F_j^{\mathrm{ML}} = \{\omega\in\Omega:\ G_{\ell_j}^{y_j}(\omega) \le y_j\},\quad \text{for}\quad j=1,\dots,K.
\]
The proposed adaptive multilevel subset simulation algorithm with a standard MC estimator for the first 
subset is summarised in Algorithm~\ref{alg:MLsubset}. 
In addition, an adaptive stopping criterion based on a statistical estimation of the rRMSE as presented in \eqref{eq:cov_subset} is introduced in Algorithm~\ref{alg:MLsubset}. Note that for simplicity we set $K=L$ and $\ell_j  = j$, for all $j=1,\dots,L$.

To employ classical subset simulation for the estimation of $\Pr(F^{\text{ML}}_1)$ instead, it suffices to replace lines 2-7 by a single-level subset simulation estimator as defined in Algorithm~\ref{alg:classicalsus}, but on the coarsest accuracy level $\ell=1$.}

\end{appendix}

\end{document}